\documentclass{amsart}
\usepackage{amscd,amssymb}
\usepackage{amsmath,amsthm,amsfonts,srcltx}
\usepackage{enumitem}
\numberwithin{equation}{section}

\newcommand{\GL}{\operatorname{GL}}

\newcommand{\e}{{e}}
\newcommand{\nofint}{n}   
\newcommand{\nofsides}{b} 
\newcommand{\Aut}{\operatorname{Aut}}

\newlength{\halfbls}

\newcommand{\noz}{r}     
\newcommand{\mult}{\mu}  
\newcommand{\prop}{\operatorname{P}}    

\newcommand{\cyl}{\mathit{cyl}}   

\renewcommand{\epsilon}{\varepsilon}

\newcommand{\SLR}{\operatorname{SL}(2,{\mathbb R})}
\newcommand{\GLR}{\operatorname{GL}^+(2,{\mathbb R})}

\newcommand{\CP}{{\mathbb C}\!\operatorname{P}^1}

\newcommand\C{\mathbb C}
\newcommand\N{\mathbb N}
\newcommand\Q{\mathbb Q}
\newcommand\R{\mathbb R}

\newcommand\ZZ{{\mathbb Z}/2{\mathbb Z}}
\newcommand\Z{\mathbb Z}

\newcommand{\cD}{{\mathcal D}}
\newcommand{\cI}{{\mathcal I}}
\newcommand{\cH}{{\mathcal H}}
\newcommand{\cL}{{\mathcal L}}
\newcommand{\cLH}{{\mathcal L}}
\newcommand{\cLQ}{{\mathcal L}}

\newcommand{\cN}{{\mathcal N}}
\newcommand{\cQ}{{\mathcal Q}}

\newcommand{\card}{\operatorname{card}}
\newcommand{\Vol}{\operatorname{Vol}}

\newtheorem{Theorem}{Theorem}[section]
\newtheorem{thm}[Theorem]{Theorem}
\newtheorem{theorem}[Theorem]{Theorem}

\newtheorem{Proposition}[Theorem]{Proposition}
\newtheorem{proposition}[Theorem]{Proposition}

\newtheorem{Lemma}[Theorem]{Lemma}
\newtheorem{lemma}[Theorem]{Lemma}

\newtheorem{Conjecture}[Theorem]{Conjecture}
\newtheorem{Question}[Theorem]{Question}

\newtheorem{Corollary}[Theorem]{Corollary}
\newtheorem{condCorollary}[Theorem]{Conditional Corollary}

\newtheorem*{Problem}{Problem}

\theoremstyle{definition}
\newtheorem{Definition}[Theorem]{Definition}
\newtheorem{definition}[Theorem]{Definition}

\newtheorem{Convention}[Theorem]{Convention}

\theoremstyle{remark}
\newtheorem{Remark}[Theorem]{Remark}
\newtheorem*{NNRemark}{Remark}
\newtheorem{rem}[Theorem]{Remark}
\newtheorem{Example}[Theorem]{Example}
\newtheorem{example}[Theorem]{Example}

\allowdisplaybreaks

\newcommand{\twopartdefotherwise}[3]
{
	\left\{
		\begin{array}{ll}
			#1 & \mbox{if } #2 \\
			#3 & \mbox{otherwise.}
		\end{array}
	\right.
}

\dedicatory{
In memory of Jean-Christophe Yoccoz.
}

\begin{document}

\title[One-cylinder square-tiled surfaces]
{Contribution of one-cylinder square-tiled surfaces to Masur--Veech volumes}

\author[V.~Delecroix]{Vincent Delecroix}
\address{
LaBRI,
Domaine universitaire,
351 cours de la Lib\'eration, 33405 Talence, FRANCE
}
\email{20100.delecroix@gmail.com}

\author[\'E.~Goujard]{\'Elise Goujard}
\thanks{Research of the second author is partially supported  by a public grant as part of the FMJH}
\address{
Institut de Math\'ematique de Bordeaux,
Domaine universitaire,
351 Cours de la Lib\'ration, 33400 Talence, FRANCE}
\email{elise.goujard@gmail.com}

\author[P.~G.~Zograf]{Peter~Zograf}
\thanks{Research of Section~\ref{s:Alternative:counting} is supported by the RScF grant 16-11-10039.}
\address{
St.~Petersburg Department, Steklov Math. Institute, Fontanka 27,
St. Petersburg 191023, and Chebyshev Laboratory,
St. Petersburg State University, 14th
Line V.O. 29B, St.Petersburg 199178 Russia}
\email{zograf@pdmi.ras.ru}

\author[A.~Zorich]{Anton Zorich}
\address{
Center for Advanced Studies, Skoltech;
Institut de Math\'ematiques de Jussieu --
Paris Rive Gauche,
Case 7012,
8 Place Aur\'elie Nemours,
75205 PARIS Cedex 13, France}
\email{anton.zorich@gmail.com}

\makeatletter
\let\@wraptoccontribs\wraptoccontribs
\makeatother

\contrib[With an Appendix by]{Philip Engel}
\thanks{Research of Appendix~\ref{Engel} is partially supported by NSF grant DMS-1502585.}
\address{Department of Mathematics,
Harvard University,
1 Oxford St, Cambridge, MA 02138, USA.
}
\email{engel@math.harvard.edu}

\begin{abstract}
We compute explicitly the \textit{absolute} contribution
of square-tiled  surfaces having a single horizontal
cylinder  to the Masur--Veech volume of any ambient
stratum  of  Abelian  differentials. The resulting count
is particularly simple and efficient in the large genus
asymptotics. Using the recent results of Aggarwal and of
Chen--M\"oller--Zagier on the long-standing conjecture
about the large genus asymptotics of Masur--Veech
volumes, we derive that the \textit{relative}
contribution is asymptotically of the order $1/d$, where
$d$ is the dimension of the stratum.

Similarly, we evaluate the
contribution of one-cylinder square-tiled surfaces to
Masur--Veech volumes of low-dimensional strata in the moduli space of
quadratic differentials. We combine this count with our
recent result on equidistribution of one-cylinder
square-tiled surfaces translated to the language of
interval exchange transformations to compute empirically
approximate values of the Masur--Veech volumes of strata
of quadratic differentials of all small dimensions.
\end{abstract}

\maketitle
\tableofcontents

\section*{Introduction}

\noindent\textbf{Siegel--Veech constants and Masur--Veech
volumes.} One of the most powerful tools in the study of
billiards in rational polygons (including ``wind-tree''
billiards with periodic obstacles in the plane), of
interval exchange transformations and of measured
foliations on surfaces is renormalization. More precisely,
to describe fine geometric and dynamical
properties of the initial billiard, interval
exchange transformation or measured foliation, one has
to find the $\GLR$-orbit closure of the associated
translation surface in the moduli space of
Abelian (or quadratic) differentials, and study
its geometry. This approach, initiated by H.~Masur and
W.~Veech four decades ago became particularly powerful
recently due to the breakthrough theorems of
Eskin--Mirzakhani--Mohammadi~\cite{Eskin:Mirzakhani}
and~\cite{Eskin:Mirzakhani:Mohammadi} that ensure that
such $\GLR$-orbit closure is linear.

The moduli space of Abelian (or quadratic) differentials is
stratified by the degrees of zeroes of the Abelian (or
quadratic) differential. Each stratum is endowed with a
natural measure, the \textit{Masur--Veech} measure, that is
preserved by the $\SLR$-action (the action by scalar
matrices rescales the volumes and only preserves the
projective class of the measure).

The Masur--Veech measure of each
connected component of a stratum is infinite. However,
passing to a level hypersurface of the function
$\frac{i}{2}\int_C\omega\wedge\bar\omega$, where $\omega$
is an Abelian differential, and $C$ is the undelying
complex curve (respectively to the level hypersurface
of the function $\int_C |q|$, where $q$ is a quadratic
differential), the Masur--Veech measure induces an
$\SLR$-invariant measure which by the results of
Masur~\cite{Masur:82} and~\cite{Veech:Gauss:measures} is
finite and ergodic.

In many important
situations the $\GLR$-orbit closure of a translation
surface is an entire connected component of a stratum. In order
to count the growth rate for the number of closed geodesics on a
translation surface as in~\cite{Eskin:Masur}, or to
describe the deviation spectrum of a measured foliation as
in~\cite{Forni:Deviation}, \cite{Zorich:How:do}, or to count the
diffusion rate of a wind-tree as in~\cite{Delecroix:Hubert:Lelievre},
\cite{Delecroix:Zorich}, one has to compute the
corresponding \textit{Siegel--Veech constants},
see~\cite{Veech:SV}, and the Lyapunov exponents of the Hodge bundle
over the connected component of stratum. Both quantities are
expressed by explicit combinatorial formulas in terms of
the Masur--Veech volumes  of  the  strata,
see~\cite{Eskin:Masur:Zorich}, \cite{EKZ}, \cite{AEZ:genus:0},
\cite{Goujard:volumes}.
\medskip

\noindent\textbf{Equidistribution of square-tiled
surfaces.} The Masur--Veech volumes of strata of Abelian
differentials and of meromorphic quadratic differentials
with at most simple poles were computed
in~\cite{Eskin:Okounkov:Inventiones},
\cite{Eskin:Okounkov:pillowcase}, and
\cite{Eskin:Okounkov:Pandharipande}. The underlying idea
(see also~\cite{Zorich:square:tiled}) was a computation of
the asymptotic number of ``integer points'' (the ones
having coordinates in $\Z\oplus i\Z$ in period coordinates)
in appropriate bounded domains exhausting the stratum. Such
integer points are represented by square-tiled surfaces. In
the case of Abelian (respectively quadratic differentials),
a square-tiled is a surface tiled by $1 \times 1$ unit
squares (resp. $1/2 \times 1/2$ unit squares). In the
Abelian case, such surface can equivalently be viewed as a
ramified cover over the square torus ramified only over
$\{0\}$ and the degree of the cover correponds to the
number of squares. In the quadratic case, a square tiled
surface is a covering of the pillowcase in $\CP$ ramified
over four points but the degree does not coincide with the
number of squares in general (there might be a factor $1$,
$2$ or $4$). Rescaling square-tiled surfaes by $\epsilon$
we get a sequence of grids that equidistribute towards the
Masur--Veech measure.

Each square-tiled surface carries
interesting combinatorial geometry, for example, the decomposition
into maximal flat horizontal cylinders. We recall in
Theorem~\ref{th:equidistribution} of
section~\ref{s:Main:results} our recent result
from~\cite{DGZZ:equidistribution} telling that square-tiled
surfaces having fixed combinatorics of horizontal cylinder
decomposition and tiled with squares of size $\epsilon$
become asymptotically equidistributed in the ambient
stratum as $\epsilon$ tends to zero. This result gives sense to the notion of
(asymptotic) probability $\prop_k$ for a ``random'' square-tiled
surface in a given stratum to have a fixed number
$k \in \{1, 2, \dots, g+\noz-1\}$ of maximal cylinders
in its horizontal decomposition, where
$g$ is the genus of the surface and $\noz$ is the number of conical
singularities.

An interval exchange transformation (or linear involution) is
called \textit{rational} if all its intervals under exchange have
rational lengths. All orbits of such interval exchange transformation
are periodic. We state in Theorem~\ref{th:iet} an
analogous equidistribution statement for rational interval
exchange transformations
(see~\cite{DGZZ:equidistribution} for the
proof) and the proportions that appear in this
context are the same as the ones for square-tiled surfaces.
The (asymptotic) probability that
a ``random'' rational interval exchange
transformation with a given permutation has $k$ maximal
bands of fellow-travelling closed trajectories
is $\prop_k$.
\medskip

\noindent\textbf{Contribution of $1$-cylinder square-tiled
surfaces and large genus asymptotics of Masur--Veech
volumes.} The only currently known computation of
Masur--Veech volumes of strata of Abelian differentials is
based on counting square-tiled surfaces. In
section~\ref{s:contribution:of:1:cylinder} we compute the
\textit{absolute} contribution $c_1(\cLH)$ of $1$-cylinder
square-tiled surfaces to the Masur--Veech volume of a
stratum $\cLH$, where $c_1(\cLH) := \prop_1(\cLH) \cdot
\Vol\cLH$. We define $c_k(\cLH)$ similarly for the absolute
contribution of $k$-cylinders square-tiled surfaces. By
definition, $\Vol\cLH = c_1(\cLH) + c_2(\cLH) + \ldots +
c_{g+r-1}(\cLH)$. We give simple close exact formulas for
the contribution $c_1(\cLH)$ to the volumes $\Vol\cH(2g-2)$
and $\Vol\cH(1,\dots,1)$ of minimal and principal strata of
Abelian differentials. We also provide sharp upper and
lower bounds for contributions of $1$-cylinder square-tiled
surfaces to the Masur--Veech volumes of any stratum of Abelian
differential. The ratio of the upper and lower bounds tends to $1$ as
$g\to+\infty$ uniformly for all strata in genus $g$, so the
bounds are particularly efficient in large genus
asymptotics.

Using the result~\cite{Chen:Moeller:Zagier} of
Chen--M\"oller--Zagier and more general
result~\cite{Aggarwal} of Aggarwal on the Masur--Veech
volume asymptotics conjectured in~\cite{Eskin:Zorich} we
prove that the corresponding \textit{relative} contribution
$\prop_1(\cLH)$ of $1$-cylinder square-tiled surfaces to
the Masur--Veech volume $\Vol\cLH$ of any stratum $\cLH$ of
Abelian differentials is asymptotically of the order $1/d$
as $g$ (equivalently $d$) tends to infinity. Here $d$ is
the dimension $d=\dim_\C(\cLH)$ of the stratum $\cLH$.
\medskip

\noindent\textbf{Siegel--Veech constants and Masur--Veech volumes of
strata of meromorphic quadratic differentials.}
The Masur--Veech volumes
of any connected component of stratum of Abelian differentials in
genus $g$ has the form $s \cdot\pi^{2g}$, where $s$ is some
rational number~\cite{Eskin:Okounkov:Inventiones}. The generating
functions in~\cite{Eskin:Okounkov:Inventiones} were translated by A.~Eskin
into computer code, which allowed to evaluate
explicitly volumes of all connected components of all strata of
Abelian differentials in genera up to $g=10$ (that is, to compute
explicitly the corresponding rational numbers $s$), and for some
strata up to $g=60$. The recent results of D. Chen, M.~M\"oller and
D.~Zagier~\cite{Chen:Moeller:Zagier} allows to compute $s$ for the
principal stratum up to genus $g=2000$ and higher.

In the \textit{quadratic} case, the Masur--Veech volume still
has the same arithmetic form $s \cdot \pi^{2 \widehat{g}}$
where $\widehat{g}$ is the so-called effective
genus~\cite{Eskin:Okounkov:pillowcase},~\cite{Eskin:Okounkov:Pandharipande}.
The computation of $s$ in the quadratic case
had to wait for a decade to be translated into
tables of numbers. One of the reasons for such a delay is a more
involved combinatorics and multitude of various conventions and
normalizations required in volume computations (which is a common
source of mistakes in normalization factors like powers
of $2$). This is why it is necessary to test theoretical predictions on some table
of volumes obtained by an independent method. In the case of Abelian
differentials, the volumes of several low-dimensional strata were
computed by a direct combinatorial method elaborated by A.~Eskin,
M.~Kontsevich and A.~Zorich; this approach is described
in~\cite{Zorich:square:tiled}. Another, even more reliable test was
provided by computer simulations of Lyapunov exponents and their ties
with volumes through Siegel--Veech constants. In the case of
quadratic differentials, explicit values of volumes of the strata
in genus zero were conjectured by M.~Kontsevich about fifteen
years ago. The conjecture was proved in recent
papers~\cite{AEZ:Dedicata} and~\cite{AEZ:genus:0}. Further explicit
values of volumes of all low-dimensional strata up to dimension $11$
were obtained in~\cite{Goujard:volumes}.

Our counting results combined with the equidistribution
Theorems~\ref{th:equidistribution} and~\ref{th:iet} allow to compute
approximate values of volumes of the strata. The idea is to evaluate
experimentally the approximate value of the probability $\prop_1(\cLQ)$ to get
a $1$-cylinder square-tiled surface taking a ``random'' square-tiled surface in
a given stratum $\cLQ$ of quadratic differentials. Then we compute rigorously
the absolute contribution $c_1(\cLQ)$ of $1$-cylinder square-tiled surfaces to
the Masur--Veech volume $\Vol\cLQ$ of the stratum. The relation $c_1(\cLQ)=
\prop_1(\cLQ) \cdot \Vol\cLQ$ now provides the approximate value of the
Masur--Veech volume $\Vol\cLQ$ of the stratum $\cLQ$ of quadratic
differentials.

This approach is completely independent of the one of
A.~Eskin and A.~Okounkov based on the representation
theory of the symmetric group. The approximate data based on this
approach were used for ``debugging'' rigorous formulas
in~\cite{Goujard:SV} and~\cite{Goujard:volumes}.

The fact that our experimental results match theoretical
ones in~\cite{AEZ:Dedicata}, \cite{AEZ:genus:0}, and
in~\cite{Goujard:volumes}, and that the theoretical values of
Siegel--Veech constants obtained in~\cite{Goujard:SV} match
independent computer experiments evaluating the Lyapunov exponents of
the Hodge bundle along the Teichm\"uller geodesic flow, as well as the
exact values of the sums of such Lyapunov exponents computed
in~\cite{Chen:Moeller} for the non-varying strata provides some
reliable evidence that the nightmare of various combinatorial
conventions leads, nevertheless, to correct and coherent
general formulas presented in~\cite{Goujard:SV} and
in~\cite{Goujard:volumes}.
\medskip

\noindent\textbf{Structure of the paper.}
In Section~\ref{s:Main:results}, we recall necessary
equidistribution results from~\cite{DGZZ:equidistribution}.
Then, in Section~\ref{s:contribution:of:1:cylinder}, we study the
contribution of $1$-cylinder square-tiled surfaces to the
Masur--Veech volumes of the strata.

Section~\ref{s:Alternative:counting} is independent of the first two:
it presents two alternative approaches to counting $1$-cylinder
square-tiled surfaces based on recursive
relations (section~\ref{ss:recursive:relations}) and on construction of
the Rauzy diagrams (section~\ref{ss:Rauzy:classes}).

The content of Appendix~\ref{s:contibution:of:diag:for:two:lattices}
was isolated to avoid overloading the main body of the paper. It
describes certain subtlety related to normalization of the Masur--Veech
volumes which is not visible in quantitative considerations, but which is
relevant and non-trivial in the context of the current paper.

Appendix~\ref{Engel} written by
Philip~Engel provides alternative proofs of results in
section~\ref{s:contribution:of:1:cylinder} based on
character theory of the symmetric group. In particular, it
provides alternative approach to the count of
$\prop_1(\cLH)$ in large genus asymptotics.

\medskip

\noindent\textbf{Acknowledgements.} We thank A.~Eskin,
C.~Matheus, L.~Monin and P.~Pushkar Jr. for numerous
valuable conversations and MPIM in Bonn for stimulating
atmosphere. We are grateful to J.~Athreya for helpful
suggestions which allowed to improve the presentation.


\section{Equidistribution}
\label{s:Main:results}

In this section we recall the recent equidistribution
results from~\cite{DGZZ:equidistribution} essential for the sequel.
We present them here not in the most general form, but in the way
which is better adapted to the context of the current paper.

\subsection{Strata of Abelian differentials}

We now introduce strata of Abelian differentials. For a more detailed
introduction, the reader might want to consult the references~\cite{Forni:Matheus:intro}
and~\cite{Zorich:flat:surfaces}.

Given a collection of non-negative integers $(m_1, \ldots,
m_\noz)$ so that $m_1 + \ldots + m_\noz = 2g-2$ we consider
the stratum of Abelian differentials $\cH(m_1, \ldots,
m_\noz)$. We fix a topological surface $S$ of genus $g$ and
$\noz$ distinct points $P_1$, \ldots, $P_\noz$ on $S$. An
element in $\cH(m_1, \ldots, m_\noz)$ is a triple $(X,
\omega, \phi: S \to X)$ where $X$ is a Riemann surface,
$\omega$ is a non-zero Abelian differential, $\phi$ is a
homeomorphism such that $\omega$ has a zero of order $m_i$
at the point $\phi(P_i)$ and does not vanish on the
complement of the set $\{P_1,\dots,P_\noz\}$. Two triples
$(X, (P_1, \ldots, P_\noz), \phi)$ and $(X', (P'_1, \ldots,
P'_\noz), \phi')$ are considered as equivalent if there is
a homeomorphism $f: S \to S$ such that $\phi' \circ f \circ
\phi^{-1}: X \to X'$ is an isomorphism of Riemann surfaces
that maps $\omega$ to $\omega'$.

The stratum $\cH(m_1, \ldots, m_\noz)$ is locally modeled
on the relative cohomology space $H^1(S, \{P_1, \ldots,
P_\noz\}; \C)$ (via the period map). Each stratum is a
$\operatorname{PL}$ complex orbifold of dimension
$2g+\noz-1$; it has at most three connected components that
have been classified in~\cite{Kontsevich:Zorich}.

Let $\cLH=\cH^{\mathit{comp}}(m_1,\dots, m_\noz)$ be a
connected component of a stratum of Abelian differentials;
denote by $d$ its complex dimension. Let
$\cLH_{\Z}\subset\cLH$ be the square-tiled surfaces
in $\cLH$, that is translation surfaces
represented in period coordinates by integer points, i.e.
by points in $H^1(S,\{P_1,\dots, P_\noz\};\Z\oplus i\Z)$.
A square-tiled surface is equivalently defined as a
translation surface tiled with unit squares.
Let $\cLH_{\Z}(N)\subset\cLH_{\Z}$
be the subset of square-tiled surfaces tiled with at most
$N$ unit squares. The Masur--Veech volume $\Vol\cLH$ of
$\cLH$ can be defined as the following limit:
\begin{equation}
\label{eq:volume:as:limit}
\Vol\cLH :=2d\cdot \lim_{N\to+\infty}
\frac{\card\cLH_\Z(N)}{N^d}\,,
\end{equation}
The existence of a finite limit was proved by
H.~Masur~\cite{Masur:82} and W.~Veech~\cite{Veech:Gauss:measures}.

\begin{NNRemark}
Consider a ``unit ball'' in $\cH^{\mathit{comp}}(m_1,\dots,
m_\noz)$ defined as the subset of translation surfaces of
area at most $1$. Geometrically, the above limit represents
the volume of this unit ball computed with respect to the
Masur--Veech volume form. The dimensional factor $2d$ is
responsible for passing from the ``volume of the unit
ball'' to the ``area of the unit sphere''. The quantity
$\Vol\cLH$ defined in equation~\eqref{eq:volume:as:limit} is
denoted in most of the papers by
$\Vol\cH_1^{\mathit{comp}}(m_1,\dots, m_\noz)$ to insist
that one passes to a hypersurface in the ambient stratum;
the total Masur--Veech volume of any stratum is, obviously,
infinite.
\end{NNRemark}

Every square-tiled surface in a stratum $\cH(m_1,\dots,
m_\noz)$ of Abelian differentials admits the decomposition
into maximal cylinders filled with closed horizontal
trajectories. By the result of J.~Smillie the number of
cylinders varies from $1$ to $g+\noz-1$ (see~\cite{Naveh}).
The set $\cLH_\Z$ can be decomposed into disjoint union of
subsets $\cLH_{\Z,k}$
$$
\cLH_\Z =\bigsqcup_{k=1}^{g+\noz-1} \cLH_{\Z,k}
$$
of respectively $k=1,2,\dots,(g+\noz-1)$-cylinder
square-tiled surfaces. Corollary 1.12
in~\cite{DGZZ:equidistribution} implies that the following
limits are well-defined for any $k$:
$$
c_k(\cLH):=2d\cdot\lim_{N\to+\infty}
\frac{\card\cLH_{\Z,k}(N)}{N^d}\,.
$$
Thus,
$$
\Vol\cLH = \sum_{k=1}^{g+\noz-1} c_k(\cLH)\,.
$$

We also introduce relative analogs of the above quantities,
namely,
\begin{equation}
\label{eq:prob}
\prop_k(\cLH):=\frac{c_k(\cLH)}{\Vol\cLH}=\lim_{N\to+\infty}
\frac{\card\cLH_{\Z,k}(N)}{\card\cLH_\Z(N)}\,.
\end{equation}
The quantity $\prop_k(\cLH)$
can be interpreted as the asymptotic frequency of $k$-cylinder
square-tiled surfaces among all square-tiled surfaces
of large bounded area in a given connected component $\cLH$
of the stratum.

Any stratum of Abelian differentials admits the natural
action of $\R_+$. For any $T>0$ and any subset $U$ of the
stratum we denote by $T\cdot U$ the subset obtained by
proportional rescaling of all translation surfaces in $U$
by the linear factor $T$, or, equivalently, by multiplying
the corresponding holomorphic $1$-form by $T$. The following
results states that each $\cLH_{\Z,k}$ equidistribute with
respect to the Masur--Veech measure.

\begin{Theorem}[\cite{DGZZ:equidistribution}]
\label{th:equidistribution}
For any non-empty relatively compact open domain $U$ in any connected component
$\cLH$ of any stratum of Abelian differentials the following limit exists
\begin{equation}
\label{eq:equidistribution:stratum}
\lim_{T\to+\infty}
\frac{\card\big((T\cdot U)\cap\cLH_{\Z,k}\big)}
{\card\big({(T\cdot U)\cap\cLH_\Z}\big)}
=
\prop_k(\cLH)
\end{equation}
and is independent of the choice of $U\subset\cLH$.
\end{Theorem}

\smallskip
We now turn to an analogue of
Theorem~\ref{th:equidistribution} for interval exchange
transformations. We say that a permutation $\pi$ on
$\{1,2,\ldots,d\}$ is {\em irreducible} if it does not
admit any $\pi$-invariant subset of the form $\{1, 2,
\dots, k\}$ where $1 \le k < d$. Given any interval
exchange transformation associated to an irreducible
permutation $\pi$ one can realize a suspension over it as a
vertical flow on an appropriate translation surface $S$.
Though the translation surface $S$ itself is not uniquely
defined, the connected component $\cLH$ of the ambient
stratum of Abelian differentials is uniquely determined by
the initial irreducible permutation. Note that in general
such stratum might have marked points in addition to
zeroes.

The space of all interval exchange transformations corresponding to a
fixed irreducible permutation $\pi$ of $d$ elements is naturally
parameterized by the lengths of $d$ intervals under exchange, so
the set of all possible interval exchange transformations with a given permutation
$\pi$ is in the natural bijective correspondence with the points of
$\R_+^d$.

If the lengths of all subintervals are integer, that is in
$\N^k$, then all orbits of the correspondent interval
exchange transformation are periodic. Equivalently, all
leaves of the vertical foliation on any suspension surface
$S$ are closed. Denote by $k$ the number of maximal
cylinders filled with such closed vertical trajectories on
$S$. This number is same for all suspension surfaces over a
given interval exchange transformation; it can be seen as
the number of bands of isomorphic fellow-travelling closed
trajectories passing through half-integer points. Denote by
$\cI_k(\pi)\subset\N^d$ the subset of integer lengths of
subintervals for which the interval exchange transformation
with given irreducible permutation $\pi$ has exactly $k$
maximal bands of trajectories. By definition,
$$
\N^d=\sqcup_k \cI_k(\pi)\,.
$$

We have
the natural action of $\R_+$ on the space of interval exchanges:
given a strictly positive number $T$ we can
rescale the lengths of all subintervals by the same factor
$T$.

\begin{Theorem}
\label{th:iet}
Given any irreducible
permutation $\pi$, let $\cLH$ be the associated connected component
of the stratum of Abelian differentials ambient for suspensions over
interval exchange transformations with permutation $\pi$.
Consider any non-empty relatively compact open domain $V$ in $\R^d_+$.
Then the following limit exists
\begin{equation}
\label{eq:equidistribution:iet}
\lim_{T\to+\infty}
\frac{\card\big((T\cdot V)\cap\cI_k(\pi)\big)}
{\card\big({(T\cdot V)\cap\N^d}\big)}
=
\lim_{T\to+\infty}
\frac{\card\big((T\cdot V)\cap\cI_k(\pi)\big)}
{\Vol_{\mathit{Eucl}}(V)\cdot T^d}
=
\prop_k(\cLH)\,,
\end{equation}
and is independent of the choice of $V\subset\R_+^d$. Here
$\Vol_{\mathit{Eucl}}(V)$ is the Euclidian volume of
$V\subset\R_+^d$.
\end{Theorem}

\subsection{Strata of quadratic differentials}

The situation with the strata in the moduli space of
meromorphic quadratic differentials with at most simple
poles is analogous (and, can more generally
be extended to any $\GLR$-invariant suborbifolds
defined over $\Q$, see~\cite{Wright:field:of:def} for
the definition). We describe here the necessary adjustments.

Recall that applying the canonical double cover $p:\hat
S\to S$ to every half-translation surface $S$ in a stratum
of meromorphic quadratic differentials with at most simple
poles we obtain a linear $\GLR$-invariant suborbifold $\hat\cLQ$
located already in the stratum of Abelian differentials
ambient for $\hat S$. Here $p$ is the double cover such
that the induced quadratic differential $p^\ast q$ is a
square of globally defined holomorphic 1-form. The stratum
of quadratic differentials is modeled on the subspace
$H^1_-(\hat S, \{\hat{P}_1,\dots, \hat{P}_\noz\};\C)$
antiinvariant under the canonical involution of $\hat S$.

The first adjustment is the convention on the normalization of the
Masur--Veech volume element in period coordinates.

\begin{Convention}
\label{conv:lattice}
We chose as a distinguished lattice in $H^1_-(\hat
S,\{\hat{P}_1,\ldots,\hat{P}_\noz\};\C{})$ the subset of those
linear forms which take values in $\Z\oplus i\Z$ on $H^-_1(\hat
S,\{\hat P_1,\dots,\hat P_\noz\};\Z)$.
\end{Convention}

Let $\cLQ=\cQ^{\mathit{comp}}(d_1,\dots, d_k)$ be a
connected component of a stratum of meromorphic quadratic
differentials with at most simple poles; denote by $d$ its
complex dimension. Let $\cLQ_\Z\subset\cLQ$ be the subset
of half-translation surfaces represented in period coordinates by lattice
points in the sense of the above Convention. Geometrically
they correspond to square-tiles surfaces tiled with squares
with side $\frac{1}{2}$. Let $\cLQ_\Z(N)\subset\cLQ_\Z$ be
the subset of square-tiled surfaces tiled with at most $N$
such squares. The Masur--Veech volume $\Vol\cLQ$ of $\cLQ$
can be defined as the following limit:
\begin{equation}
\label{eq:volume:as:limit:quadratic}
\Vol\cLQ :=2d\cdot 2^d\cdot \lim_{N\to+\infty}
\frac{\card\cLQ_\Z(N)}{N^d}
=
2d\cdot \lim_{N\to+\infty}
\frac{\card\cLQ_\Z(2N)}{N^d}
\,,
\end{equation}
\begin{NNRemark}
The extra factor $2^d$
in~\eqref{eq:volume:as:limit:quadratic} compared
to~\eqref{eq:volume:as:limit} has the following origin.
The cover $\hat S$ belongs to the ``unit ball'' in $\hat\cLQ$ if and
only if the initial half-translation surface $S$ has area
at most $1/2$. On the other hand, now the squares of tiling
have area $\frac{1}{4}$ and not unit area as before.
\end{NNRemark}

Now in complete analogy with the case of Abelian differentials
we define the subset $\cLQ_{\Z,k}\subset\cLQ_\Z$ of
square-tiled surfaces having exactly $k$ maximal horizontal cylinders
and the subset $\cLQ_{\Z,k}(N)\subset\cLQ_{\Z,k}$ of those of them which are
tiled with at most $N$ squares (with side $1/2$).
Results from~\cite{DGZZ:equidistribution} imply that
for any $\cLQ$ and any $k$ there are
well defined contributions of $k$-cylinder
square-tiled surfaces to the Masur--Veech volume of $\cLQ$:
$$
c_k(\cLQ):=2d\cdot2^d\cdot\lim_{N\to+\infty}
\frac{\card\cLQ_{\Z,k}(N)}{N^d}\,.
$$
As before,
$$
\Vol\cLQ = \sum_{k=1}^{\widehat{g}+\noz-1} c_k(\cLQ)\,.
$$
A complete analog of Theorem~\ref{th:equidistribution} holds
for the asymptotic proportions
\begin{equation}
\label{eq:prob:Q}
\prop_k(\cLQ):=\frac{c_k(\cLQ)}{\Vol\cLQ}=\lim_{N\to+\infty}
\frac{\card\cLQ_{\Z,k}(N)}{\card\cLQ_\Z(N)}\,.
\end{equation}

The second adjustment concerns interval exchange
transformations. An irreducible permutation is replaced now
by by an \textit{irreducible generalized permutation} $\pi$
of $d+1$ elements, where $d=\dim_{\C}\cLQ$; see
combinatorial Definition~3.1 in~\cite{Boissy:Lanneau} of
irreducibility.

The lengths $\lambda_i$ of subintervals of the
corresponding irreducible generalized interval exchange
transformation (called \textit{linear involution} in the
original paper~\cite{Danthony:Nogueira} introducing these
objects) satisfy a nontrivial linear relation of the form
\begin{equation}
\label{eq:relation:generalized:iet}
\lambda_{i_1}+\dots+\lambda_{i_r}=\lambda_{j_1}+\dots+\lambda_{j_s}\,.
\end{equation}
depending on $\pi$, where every index from the set
$\{1,\dots,d+1\}$ appears at most once, and there is at
least one term on each side of the equation. Choose any
parameter involved into
relation~\eqref{eq:relation:generalized:iet}, say,
$\lambda_{j_s}$ for definitiveness. The remaining $d$
lengths of intervals under exchange in our generalized
interval exchange transformation provide coordinates in the
space of generalized interval exchange transformations
corresponding to the irreducible generalized permutation
$\pi$. The positivity condition on the remaining length
$\lambda_{j_s}$ implies that the set of parameters is the
polyhedral cone $C^d_+(\pi)\subset\R_+^d$ obtained as the
intersection of $\R_+^d$ with the half-space defined by the
equation
$$
\lambda_{i_1}+\dots+\lambda_{i_r}-
(\lambda_{j_1}+\dots+\lambda_{j_{s-1}})>0\,.
$$

Denote by $\cI_k(\pi)\subset C^d_+(\pi)\cap(\N/2)^d$ the subset of
half-integer lengths of subintervals for which the interval exchange
transformation with the given irreducible generalized permutation
$\pi$ has exactly $k$ maximal bands of trajectories in the same sense
as above.

\begin{Theorem}
\label{th:generalized:iet}
Given any irreducible generalized
permutation $\pi$, let $\cLQ$ be the associated connected component
of the stratum of meromorphic quadratic
differentials with at most simple poles corresponding to any suspensions over
$\pi$. Consider any open relatively compact domain $V$ in $C^d_+(\pi)$.
Then, the following limit exists
\begin{equation}
\label{eq:equidistribution:generalized:iet}
\lim_{T\to+\infty}
\frac{\card\big((T\cdot V)\cap\cI_k(\pi)\big)}
{\card\big({(T\cdot V)\cap(\N/2)^d}\big)}
=
\lim_{T\to+\infty}
\frac{\card\big((T\cdot V)\cap\cI_k(\pi)\big)}
{2^d\cdot\Vol_{\mathit{Eucl}}(V)\cdot T^d}
=
\prop_k(\cLQ)\,,
\end{equation}
and is independent of the choice of $V\subset C_+^d(\pi)$.
Here $\Vol_{\mathit{Eucl}}(V)$ is the Euclidian volume of
$V\subset\R_+^d$.
\end{Theorem}

\begin{NNRemark}
An alternative natural choice of the lattice (in other words,
an alternative definition of "square-tiled surface")
and its effect on the quantities $\prop_k(\cLQ)$ is discussed in
Appendix~\ref{s:contibution:of:diag:for:two:lattices}.
\end{NNRemark}

\section{Contribution of $1$-cylinder square-tiled surfaces to Masur--Veech volumes}
\label{s:contribution:of:1:cylinder}

In this section we consider square-tiled surfaces
represented by a \emph{single} maximal flat
cylinder filled by closed horizontal leaves and their
contributions to the Masur--Veech volumes of strata of Abelian
differentials and of meromorphic quadratic differentials with at most
simple poles.

In section~\ref{ss:Contribution:of:1:cylinder:diagrams} we
state the main results. Their proofs are postponed to
sections~\ref{ss:contribution:of:one:1:cylinder:diagram:computation}
and~\ref{ss:1:cylinder:diagrams:Abelian}. In
section~\ref{ss:Asymptotics:in:large:genera} we apply our
results to strata of Abelian differentials in large genus
and discuss how they compare with the asymptotic behavior
of Masur--Veech volumes. In
section~\ref{ss:Application:experimental:evaluation:of:MV:volumes}
we describe the experimental approach to the computation of
Masur--Veech volumes unifying our equidistribution and
counting results.

We proceed in section~\ref{ss:contribution:of:one:1:cylinder:diagram:computation}
with a detailed discussion of relevant combinatorial aspects
and with a computation of the contribution of a single $1$-cylinder
separatrix diagram to the Masur--Veech volume of the ambient stratum
proving Propositions~\ref{pr:contribution:Abelian}
and~\ref{pr:contribution:quadratic}.

In section~\ref{ss:1:cylinder:diagrams:Abelian} we count the number
of $1$-cylinder diagrams for strata of Abelian differentials.
Combining our count with the result of
section~\ref{ss:contribution:of:one:1:cylinder:diagram:computation}
we derive very sharp
bounds~\eqref{eq:contribution:all:1:cyl:estimate} for the absolute
contribution of $1$-cylinder square-tiled surfaces to the
Masur--Veech volume claimed in
Theorem~\ref{th:contribution:all:1:cyl:estimate}. We also obtain
exact closed formulas for the absolute contributions
of $1$-cylinder square-tiled surfaces to the Masur--Veech volumes of
the minimal and principal strata stated in
Corollary~\ref{cor:total:contribution:min:and:principal}.

\subsection{Jenkins--Strebel differentials. Critical graphs (separatrix diagrams)}
\label{ss:separatrix:diagrams}

Assume that all leaves of the horizontal foliation of an
Abelian or quadratic differential are either closed or connect
critical points (a leaf joining two critical points is called a
\textit{saddle connection} or a \textit{separatrix}). Later we
will be saying simply that the horizontal foliation has only
closed leaves. The square of an Abelian differential, or a
quadratic differential having this property is called a
\textit{Jenkins--Strebel} quadratic differential,
see~\cite{Strebel}. For example, square-tiled surfaces
provide particular cases of Jenkins--Strebel differentials.

Following~\cite{Kontsevich:Zorich} we will associate
with each Abelian or quadratic differential whose
horizontal foliation has only closed leaves a
combinatorial data called \textit{separatrix diagram}
(also known as the \textit{critical graph} of a
Jenkins--Strebel differential).

We start with an informal explanation. Consider the union of all
saddle connections for the horizontal foliation, and add all
critical points. We obtain a finite graph $\Gamma$. In the case of an
Abelian differential it is oriented, where the orientation on the
edges comes from the canonical orientation of the horizontal
foliation. In both cases of an Abelian or quadratic
differential, the graph $\Gamma$ is drawn on an oriented surface,
therefore it carries a \textit{ribbon structure}, i.e. on the star
of each vertex $v$ a cyclic order is given, namely the
counterclockwise order in which half-edges are attached to $v$. In the
case of an Abelian differential, the direction of edges attached to
$v$ alternates (between directions toward $v$ and from $v$) as we
follow the cyclic order.

It is well known that any finite ribbon graph $\Gamma$ defines
canonically (up to an isotopy) an oriented surface $S(\Gamma)$
with boundary. To obtain this surface we replace each edge of
$\Gamma$ by a thin oriented strip (rectangle) and glue these
strips together using the cyclic order in each vertex of
$\Gamma$. In our case surface $S(\Gamma)$ can be realized as a
tubular $\varepsilon$-neighborhood (in the sense of the transversal
measure) of the union of all saddle connections for sufficiently
small $\varepsilon>0$.

In the case of an Abelian differential, the orientation of edges
of $\Gamma$ gives rise to the orientation of the boundary of
$S(\Gamma)$. Notice that this orientation is {\it not} the same as
the canonical orientation of the boundary of an oriented surface.
Thus, connected components of the boundary of $S(\Gamma)$ are
decomposed into two classes: positively and negatively oriented
(positively when two orientations of the boundary components
coincide and negatively, when they are opposite). We shall also
refer to them as the \textit{top} and \textit{bottom} components of
the corresponding cylinder, with respect to the positive orientation
of the vertical foliation. The complement to the tubular
$\varepsilon$-neighborhood of $\Gamma$ is a finite disjoint union of
open flat cylinders foliated by circles. It gives a
decomposition of the set of boundary circles
$\pi_0(\partial S(\Gamma))$ into pairs of components having
opposite orientation.

Now we are ready to give a formal definition (see~\S 4
in~\cite{Kontsevich:Zorich} for more details on separatrix diagrams):

\begin{Definition}
A \textit{separatrix diagram} is a (not necessarily connected) oriented
ribbon graph $\Gamma$, and a decomposition of the set of boundary components of
$S(\Gamma)$ into pairs and so that identifying these boundary components we
get a connected surface.

An \textit{orientable} separatrix diagram
satisfies the following additional properties:
\begin{enumerate}
\item the orientation of the half-edges at any vertex alternates with
respect to the cyclic order of edges at this vertex;
\item there is one positively oriented and one negatively
oriented boundary component in each pair.
\end{enumerate}
\end{Definition}

Any separatrix diagram represents a measured foliation with
only closed leaves on a compact oriented surface without boundary.
We say that a diagram is \textit{realizable} if, moreover, this
measured foliation can be chosen as the horizontal foliation of
some Abelian or quadratic differential (depending on
orientability of the foliation).

Assign to each saddle connection a real variable standing for its
``length''. Now any boundary component is also naturally endowed
with a ``length''. If we want to glue flat cylinders to the
boundary components, the lengths of the components in every pair
should match each other. Thus, for every two boundary components
paired together we get a linear
relation on the lengths of saddle connections. Clearly, a diagram is
realizable if and only if the corresponding system of linear
equations on lengths of saddle connections admits a
strictly positive solution.

As an example, consider all possible separatrix diagrams which might
appear in the stratum $\cH(2)$ (see~\S~5
in~\cite{Zorich:square:tiled} for more details). The single conical
singularity of a flat surface in $\cH(2)$ has cone angle $6\pi$, so
every separatrix diagram has a single vertex with six prongs. Since
it corresponds to the stratum of Abelian differentials, it should be
oriented. All such diagrams are presented in Figure~\ref{fig:diag}.
We see, that the left diagram $\cD_1$ defines a translation surface
with a single pair of boundary components (i.e. with a single
cylinder filled with closed horizontal leaves); it is realizable for
all positive values $\ell_1,\ell_2,\ell_3$ of length parameters. The
middle diagram defines a surface with two pairs of boundary
components (i.e. with two cylinders filled with closed horizontal
leaves); it is realizable when $\ell_1=\ell_3$. The right diagram
would correspond to a surface with a single ``top'' boundary
component, and with three ``bottom'' boundary components. Since each
``top'' boundary component must be attached to a ``bottom'' boundary
component by a cylinder, this diagram is not realizable by a
translation surface.

\begin{figure}[htb]
\includegraphics{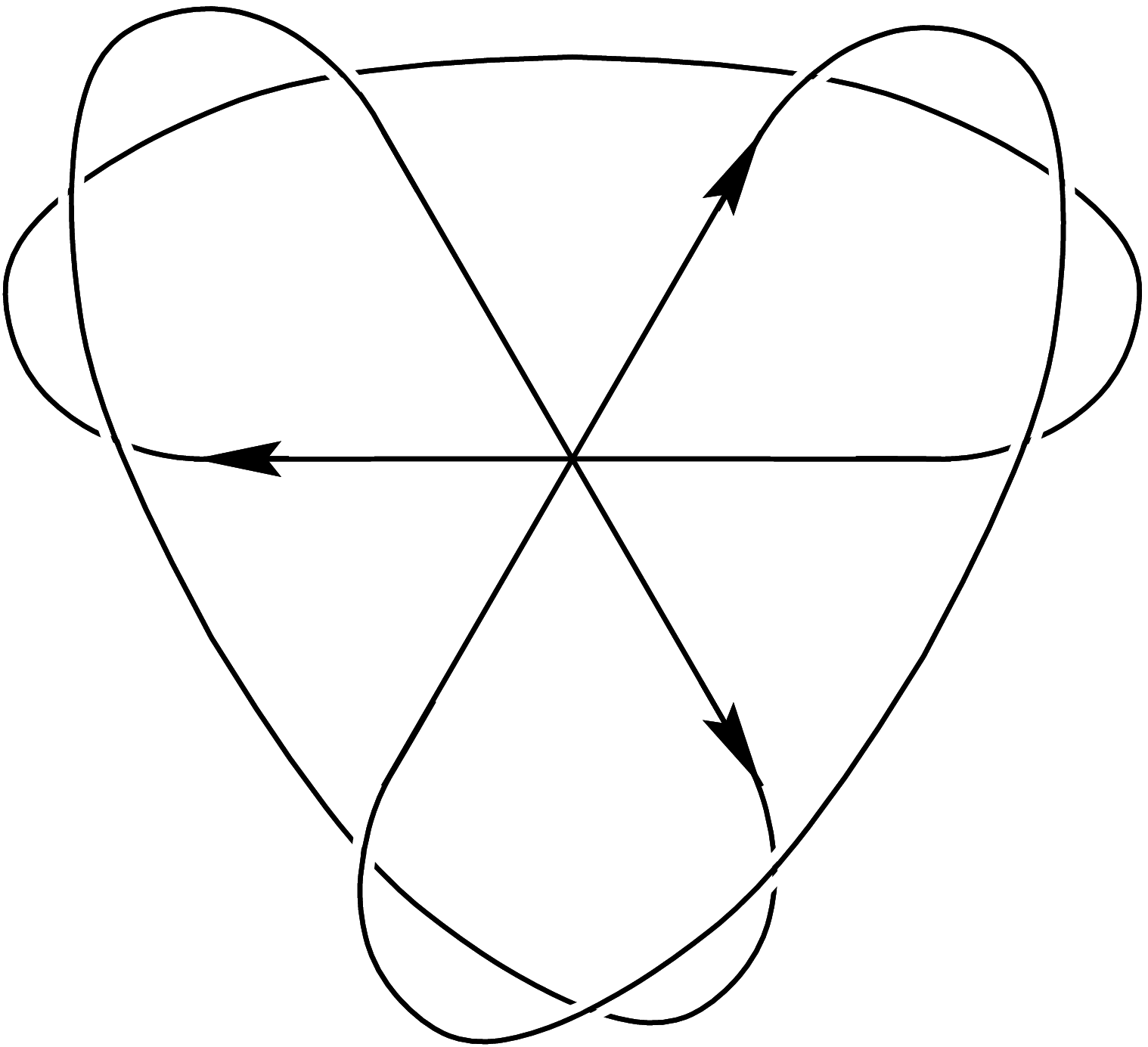}
\includegraphics{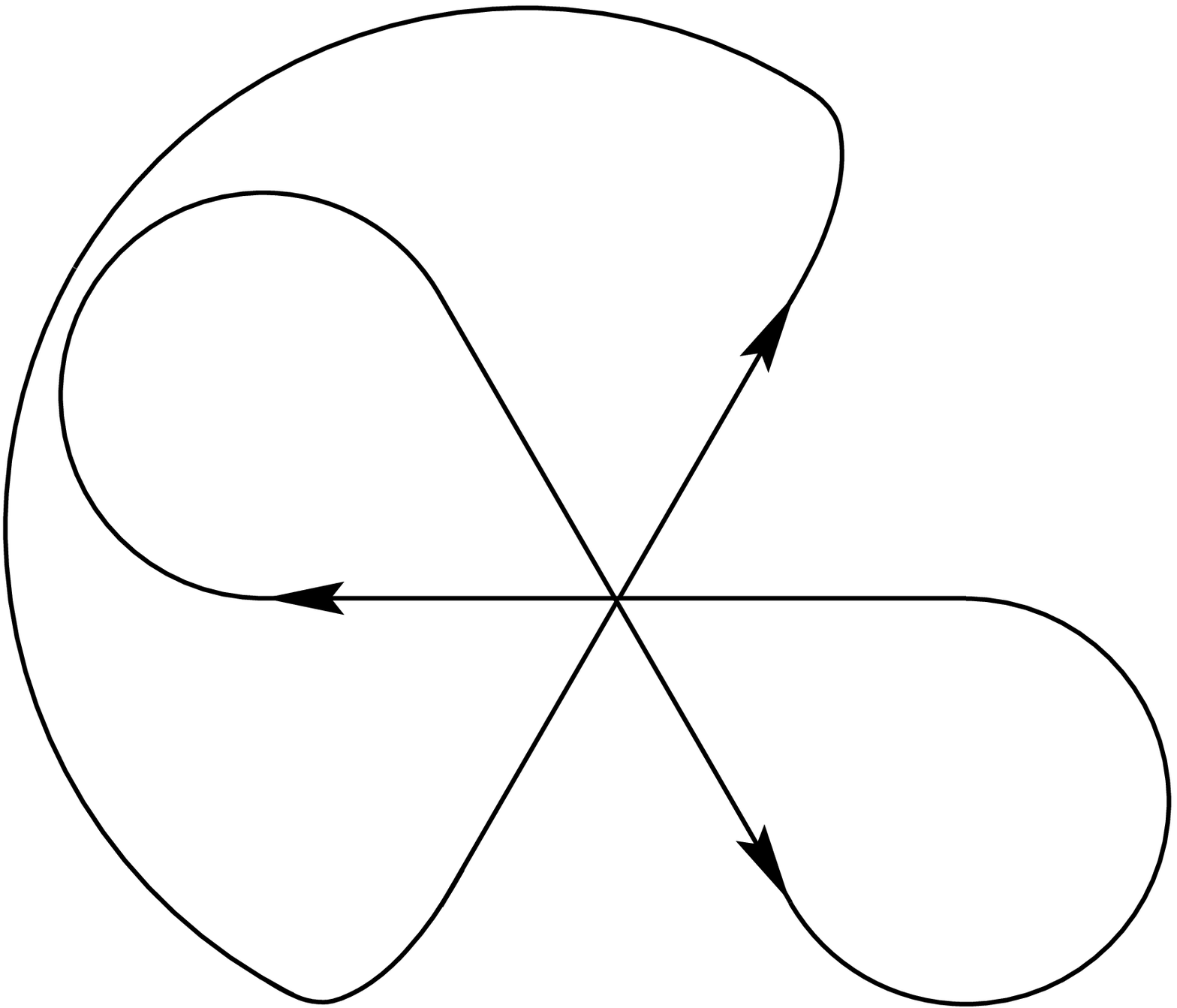}
\includegraphics{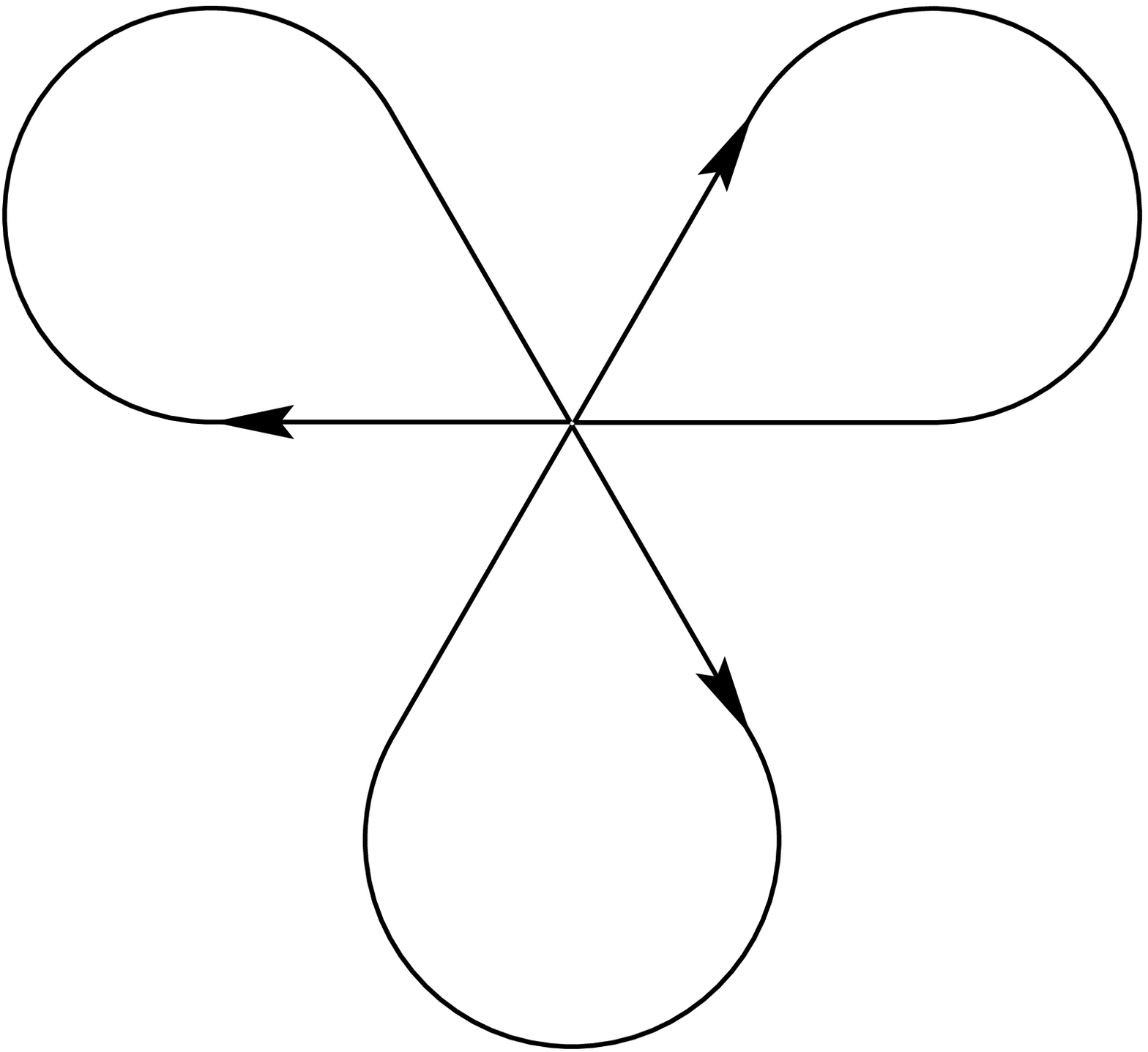}
\vspace{125bp} 
\begin{picture}(0,0)(170,-10)
\put(180,0){\begin{picture}(0,0)(0,0)
\put(-129,20){$\ell_1$}
\put(-95,75){$\ell_2$}
\put(-162,75){$\ell_3$}
\put(-12,100){$\ell_1$}
\put(35,85){$\ell_2$}
\put(-12,15){$\ell_3$}
\put(-129,-5){$\cD_1$}
\put(-12,-5){$\cD_2$}
\put(90,-5){$\cD_3$}
\end{picture}}
\end{picture}
\caption{
\label{fig:diag}
The separatrix diagrams represent from left to right a
square-tiled surface glued from: $\cD_1$ --- one cylinder; $\cD_2$
--- two cylinders; $\cD_3$ --- not realizable by a square-tiled
surface.}
\end{figure}

\subsection{Contribution of $1$-cylinder diagrams}
\label{ss:Contribution:of:1:cylinder:diagrams}

Recall from Section~\ref{s:Main:results} that
the volume of a stratum of Abelian differential $\cLH = \cH(m_1,\ldots,m_\noz)$
defined by~\eqref{eq:volume:as:limit} can be written as a sum of contributions
of 1-cylinder surfaces, 2-cylinder surfaces, etc.
$$
\Vol\cLH = c_1(\cLH)+c_2(\cLH)+\dots+c_{g+\noz-1}(\cLH)
$$
Moreover, it follows also from~\cite{DGZZ:equidistribution} that
each $c_k(\cLH)$ decomposes itself as a sum of contribution of
each diagram
$$
c_k(\cLH) =  \sum_{\text{realizable $\cD$ in $\cLH$ with $k$ cylinders}} c(\cD)\,,
$$
where  $c(\cD)=c(\cD_\Z)$ is
the contribution of a realizable separatrix diagram $\cD$.
The same decomposition holds for the strata of quadratic differentials.

\begin{Proposition}
\label{pr:contribution:Abelian}
The contribution of any $1$-cylinder orientable separatrix diagram
$\cD$ to the volume $\Vol\cH(m_1,\dots,m_\noz)$ of a
stratum of Abelian differentials equals
\begin{equation}
\label{eq:contribution:numbered}
c(\cD)=\cfrac{2}{|\Aut(\cD)|}\cdot
\cfrac{\mult_1!\cdot \mult_2! \cdots}{(d-2)!}\,\cdot \zeta(d)\,.
\end{equation}
Here $|\Aut(\cD)|$ is the order of the symmetry group of the
separatrix diagram $\cD$; $\mult_i$ is the number of zeroes of order
$i$, i.e. the multiplicity of the entry $i$ in the set
$\{m_1,\dots,m_\noz\}$; and $d=\dim_\C
\cH(m_1,\dots,m_\noz)=2g+\noz-1$. Speaking of the volume of the
stratum we assume that the zeroes $P_1,\dots,P_\noz$ of the Abelian
differentials are numbered (labeled).
\end{Proposition}

For the case of quadratic differentials, consider a
non-orientable measured foliation on a closed surface such
that all its regular leaves are closed and fill a single
flat cylinder. We refer the reader to
Figure~\ref{fig:Jenkins:Strebel} in
section~\ref{ss:contribution:of:one:1:cylinder:diagram:computation}
for an illustration. Cut the surface along all saddle
connections to unwrap it into a cylinder. Every saddle
connection is presented exactly two times on the boundary
of the resulting cylinder. Call any of the two boundary
components of the cylinder the ``top'' one and the
complementary component --- the ``bottom'' one.Denote by
$l$ the number of saddle connections which are presented
once on top and once on the bottom; by $m$ the number of
saddle connections which are presented twice on the top,
and by $n$ the number of saddle connections which are
presented twice on the bottom. It is immediate to see that
if the original flat surface belongs to some stratum
$\cQ(d_1,\dots,d_k)$ of meromorphic quadratic differentials
with at most simple poles, then $l+m+n=d$, where $d=\dim_\C
\cQ(d_1,\dots,d_k)=2g+k-2$. Since the measured foliation is
non orientable, both $m$ and $n$ are strictly positive.

\begin{Proposition}
\label{pr:contribution:quadratic}
The contribution of any $1$-cylinder separatrix
diagram $\cD$ to the volume
$\Vol\cQ(d_1,\dots,d_k)$
of a stratum of meromorphic quadratic
differentials with at most simple poles equals
\begin{equation}
\label{eq:general:contribution}
c(\cD)=
\cfrac{2^{l+2}}{|\Aut(\cD)|}\cdot
\frac{(m+n-2)!}{(m-1)!(n-1)!}\,
\cdot\cfrac{\mult_{-1}!\cdot\mult_1!\cdot \mult_2! \cdots}{(d-2)!}\,
\cdot\zeta(d)\,.
\end{equation}
Here  $|\Aut(\cD)|$ is the order of the symmetry group of the
separatrix  diagram  (ribbon graph) $\cD$; $\mult_{-1}$ is the number
of  simple  poles;  $\mult_i$  is  the number of zeroes of order $i$;
$d=\dim_\C \cQ(d_1,\dots,d_k)=2g+k-2$; $m$ and $n$ are the numbers of
saddle  connections  which  are  presented only on top (respectively, on
bottom) boundary components of the cylinder.

Defining the symmetry group $\Aut(\cD)$ we assume that none of the
vertices, edges, or boundary components of the ribbon graph $\cD$ is
labeled; however, we assume that the orientation of the ribbons is
fixed. Defining the volume of $\cQ(d_1,\dots,d_k)$ we assume that
the zeroes and poles are numbered (labeled).
\end{Proposition}

Propositions~\ref{pr:contribution:Abelian}
and~\ref{pr:contribution:quadratic} are proved in
section~\ref{ss:contribution:of:one:1:cylinder:diagram:computation}.

\smallskip
We now use the Frobenius formula in the theory of
representations of the symmetric group to count the number
of 1-cylinder diagrams in a given stratum of Abelian
differentials. As remarked in~\cite{Delecroix} the
1-cylinder diagrams can be seen as pairs of $n$-cycles
whose product belongs to a given conjugacy class determined
by the stratum. Frobenius theorem allows to interpret our
count as a sum over irreducible characters of the symmetric
group. We follow the notations of \S A.2 in~\cite{Zagier}
and refer the reader to this reference for all the relevant
background.

Recall that a representation $\rho$ of the symmetric group
$S_n$ is a homomorphism $\rho: S_n \to \GL(V)$ where $V$ is
a finite dimensional complex vector space. The simplest
example is given by the permutation action of $S_n$ on
coordinates in $\C^n$. This action leaves invariant the
1-dimensional subspace generated by the sum $e_1 + e_2 +
\ldots +e_n$ of the vectors of the basis and the
$(n-1)$-dimensional subspace $W_n := \left\{\sum x_i e_i :
\sum x_i = 0\right\}$, where $e_i$ denotes the elements of
the standard basis of $\C^n$. The representation
$\mathbf{St}_n$ induced on $W_n$ is irreducible (i.e. it
does not contain non-trivial invariant subspaces).

Now define the characters of the exterior powers of the
representation $\mathbf{St}_n$
$$
\chi_j(g):=\operatorname{tr}(g,\pi_j)\qquad
\pi_j:=\wedge^j(\mathbf{St}_n)\qquad
(0\le j\le n-1)\,.
$$

\begin{Theorem}
\label{th:contribution:all:1:cyl}
The absolute contribution $c_1(\cLH)$ of
all $1$-cylinder orientable separatrix diagrams
$\cD_\alpha$ to the volume $\Vol\cLH$
of the stratum
$\cLH=\cH(m_1,\dots,m_\noz)$ of Abelian
differentials equals
\begin{equation}
\label{eq:contribution:all:1:cyl}
c_1(\cLH)=
\frac{2}{n!}\cdot
\prod_k \frac{1}{(k+1)^{\mult_k}}\cdot
\sum_{j=0}^{n-1} j!\,(n-1-j)!\,\chi_j(\nu)
\,\cdot \zeta(n+1)\,.
\end{equation}

Here
$n=(m_1+1)+\dots+(m_\noz+1)=\dim_\C\cH(m_1,\dots,m_\noz)-1$;
$\nu\in S_n$ is any permutation which decomposes into
cycles of  lengths  $(m_1+1),\dots,(m_\noz+1)$;  $\mult_i$
is the number of zeroes of order $i$, i.e. the multiplicity
of the entry $i$ in the multiset $\{m_1,\dots,m_\noz\}$.
Speaking of the volume of the  stratum  we  assume  that
the  zeroes $P_1,\dots,P_\noz$ of the Abelian differentials
are numbered (labeled).
\end{Theorem}

\begin{Remark}
\label{rm:zeta:d}
Considering 1-cylinder square-tiled
surfaces we never restricted the height of the cylinder. In
certain context (for example, for count of meanders as
in~\cite{DGZZ:meanders}), one needs to consider only
1-cylinder square-tiled surfaces represented by a
\textit{single} horizontal band of squares. Denote by
$\cyl_1(\cL)$ the contribution to the Masur--Veech volume
$\Vol\cL$ of a stratum $\cL$ of Abelian or quadratic
differentials coming from such more specific 1-cylinder
square-tiled surfaces. It would be clear from the proofs of
Propositions~\ref{pr:contribution:Abelian}
and~\ref{pr:contribution:quadratic} that the corresponding
contributions $\cyl(\cD)$ and $c(\cD)$ of an individual
1-cylinder diagram $\cD$ to $\Vol\cL$ respectively with and
without this extra restriction on the height of the
cylinder, and hence, the contributions $\cyl_1(\cL)$ and
$c_1(\cL)$ to the volume of the stratum, differ by the
factor $\zeta(d)$, where $d=\dim_{\C}\cL$.
\begin{align*}
c_1(\cL) & = \zeta(d) \cdot \cyl_1(\cL)\,,\\
c_1(\cD) & = \zeta(d) \cdot \cyl_1(\cD)\,.
\end{align*}
\end{Remark}

Applying  Theorem~\ref{th:contribution:all:1:cyl}  to  two
particular strata,  namely  to  the principal stratum and
to the minimal one, we get         a         close
expression        given        by
Corollary~\ref{cor:total:contribution:min:and:principal}.
For  other strata
Theorem~\ref{th:contribution:all:1:cyl:estimate}     below
provides a very good
estimate~\eqref{eq:contribution:all:1:cyl:estimate} for
$c_1(\cLH)$.

\begin{Corollary}
\label{cor:total:contribution:min:and:principal}
The absolute contribution of all $1$-cylinder orientable separatrix
diagrams to the volume $\Vol\cH(1^{2g-2}) = \Vol\cH(\underbrace{1,\dots,1}_{2g-2})$
of the principal stratum and to the volume $\Vol\cH(2g-2)$ of the minimal
stratum of Abelian differentials equals
\begin{align}
\label{eq:contribution:principal}
c_1(\cH(1^{2g-2}))=
\frac{\zeta(4g-3)}{4g-2}\cdot\frac{4}{2^{2g-2}}
\\
\label{eq:contribution:minimal}
c_1(\cH(2g-2))=\frac{\zeta(2g)}{2g}\cdot\frac{4}{2g-1}
\end{align}
\end{Corollary}

Theorem~\ref{th:contribution:all:1:cyl}                           and
Corollary~\ref{cor:total:contribution:min:and:principal}  are  proved
in section~\ref{ss:1:cylinder:diagrams:Abelian}.
See also appendix~\ref{Engel} for alternative proofs.

\begin{Example}
A square-tiled
surface  in  the  stratum $\cH(2)$ may have one of the two separatrix
diagrams   $\cD_1,   \cD_2$   shown   in   Figure~\ref{fig:diag}.
Square-tiled  surfaces
corresponding  to  separatrix  diagrams $\cD_1,\cD_2$ have one and two
maximal  cylinders  filled  with  closed regular horizontal geodesics
respectively.   Direct  computations
in~\cite{Zorich:square:tiled}  show that the constants
$c(\cD_i)$ have values
\begin{equation}
\label{eq:contributions:1:2:for:H2}
 c(\cD_1)=\frac{2}{3!}\cdot \zeta(4)
\qquad\qquad
c(\cD_2)=\cfrac{2}{3!} \cdot \cfrac{5}{4} \cdot \zeta(4)\,.
\end{equation}
Note that the  separatrix  diagram  $\cD_1$  has  symmetry  of  order
$3$,  so $|\Aut(\cD_1)|=3!$,       and       the       value
$c(\cD_1)$ matches~\eqref{eq:contribution:numbered}. We have :

\[\Vol\cH(2)=c(\cD_1)+c(\cD_2)=\frac{3}{4}\zeta(4)=\frac{\pi^4}{120}.\]

Morally,    Theorem~\ref{th:equidistribution}    implies    that    a
``random''  Abelian  differential  with  rational periods in any open
subset   in   $\cH(2)$   would   have  single  maximal  horizontal
cylinder   filling   the   entire   surface  with  probability  $4/9$
and   two   horizontal  cylinders  of  different  perimeters  filling
together the   entire  surface with probability $5/9$.
\end{Example}

The  next  example  shows  that the values of similar proportions for
more complicated strata become much more elaborate.

\begin{Example}
\label{ex:H31}
A  square-tiled surface in the stratum $\cH(3,1)$ might have from $1$
to  $4$  cylinders.  Taking  the  sums  of  $c(\cD_\alpha)$  for  $4$
one-cylinder   diagrams   in   the   stratum   $\cH(3,1)$,  $30$
two-cylinder  diagrams,  $44$  three-cylinder  diagrams, and $10$
four-cylinder  diagrams  (here  the  numbers  of  oriented separatrix
diagrams are given without any weights), and computing the
proportions, or probabilities
\[\prop_i(\cH(3,1))=\frac{c_i(3,1)}{\Vol\cH(3,1)}\] we get
(see~\cite{Zorich:square:tiled}):
\begin{equation*}
\begin{split}
 \prop_1(\cH(3,1))& =
\cfrac{3\,\zeta(7)}{16\,\zeta(6)}\approx 0.19 \,,  \\  & \\ 
 \prop_2(\cH(3,1))& =
\cfrac{55\,\zeta(1,6) + 29\,\zeta(2,5) + 15\,\zeta(3,4) + 8\,\zeta(4,3) + 4\,\zeta(5,2)}
    {16\,\zeta(6)}\approx 0.47 \,, \\  & \\ 
\prop_3(\cH(3,1))&  =
\cfrac{1}{32\,\zeta(6)}
\bigg( 12\,\zeta(6) - 12\,\zeta(7) + 48\,\zeta(4)\,\zeta(1,2) + 48\,\zeta(3)\,\zeta(1,3) \\ &
  + 24\,\zeta(2)\,\zeta(1,4) + 6\,\zeta(1,5) - 250\,\zeta(1,6) - 6\,\zeta(3)\,\zeta(2,2)
\\ &
 -  5\,\zeta(2)\,\zeta(2,3) +
6\,\zeta(2,4) - 52\,\zeta(2,5) + 6\,\zeta(3,3) - 82\,\zeta(3,4)
\\ &
 + 6\,\zeta(4,2) - 54\,\zeta(4,3) + 6\,\zeta(5,2) + 120\,\zeta(1,1,5) - 30\,\zeta(1,2,4)
\\ &
  - 120\,\zeta(1,3,3) - 120\,\zeta(1,4,2) - 54\,\zeta(2,1,4) - 34\,\zeta(2,2,3)
\\ &
  - 29\,\zeta(2,3,2) - 88\,\zeta(3,1,3) - 34\,\zeta(3,2,2) - 48\,\zeta(4,1,2)
\bigg)\approx 0.30  \,, \\  & \\ 
 \prop_4(\cH(3,1))& = \cfrac{\zeta(2)}{8\,\zeta(6)}
\,\bigg( \zeta(4) - \zeta(5) + \zeta(1,3) + \zeta(2,2) - \zeta(2,3) - \zeta(3,2) \bigg)\approx 0.04 \,.
\end{split}
\end{equation*}
\end{Example}

Note  that for separatrix diagrams $\cD_\alpha$ with $k>1$ cylinders,
the  contribution  $c(\cD_\alpha)$ of the diagram varies from diagram
to  diagram, and even in the example above the contribution of an
individual diagram is not necessarily reduced
to  a  polynomial in multiple zeta values with rational coefficients.

\begin{Question}
Is it true that the total contribution of all $k$-cylinder separatrix
diagrams  to  the volume of any stratum of Abelian differentials is a
polynomial  in  multiple  zeta values with rational (or even integer)
coefficients?
\end{Question}

\subsection{Asymptotics in large genera}
\label{ss:Asymptotics:in:large:genera}

Theorem~\ref{th:contribution:all:1:cyl}   combined   with   Theorem~2
in~\cite{Zagier:bounds} provides the following result which is proved
in section~\ref{ss:1:cylinder:diagrams:Abelian}.

\begin{Theorem}
\label{th:contribution:all:1:cyl:estimate}
The absolute contribution $c_1(\cLH)$ of all
$1$-cylinder orientable separatrix diagrams to
the volume $\Vol\cLH$
of any stratum $\cLH=\cH(m_1,\dots,m_\noz)$
of Abelian differentials satisfies the following bounds
\begin{equation}
\label{eq:contribution:all:1:cyl:estimate}
\frac{\zeta(d)}{d+1}\cdot
\frac{4}{(m_1+1)\dots(m_\noz+1)}
\le c_1(\cLH)
\le
\frac{\zeta(d)}{d-\frac{10}{29}}
\cdot
\frac{4}{(m_1+1)\dots(m_\noz+1)}
,
\end{equation}
where $d=\dim_\C\cH(m_1,\dots,m_\noz)$.
\end{Theorem}

To   discuss   the   asymptotic  behavior  of  the
\textit{relative} contribution $\prop_1(\cLH)$ for the
strata of large genera  we use recent result of
A.~Aggarwal~\cite{Aggarwal} on the Masur--Veech volume
asymtptotics. Let $\Pi_{2g-2}$ be the set of integer
partitions $m=(m_1, \ldots,m_\noz)$ of $2g-2$ into
(unordered) positive numbers.

\begin{Theorem}[\cite{Aggarwal}]
\label{conj:vol}
For any $m\in\Pi_{2g-2}$ one has
\begin{equation}
\label{eq:asymptotic:formula:for:the:volume}
\Vol\cH(m_1,\dots,m_\noz)=\cfrac{4}{(m_1+1)\cdot\dots\cdot(m_\noz+1)}
\cdot (1+\varepsilon_1(m)),
\end{equation}
where
$$
\lim_{g\to\infty} \max_{m\in\Pi_{2g-2}} \varepsilon_1(m) = 0.
$$
\end{Theorem}

The above result was a long standing conjecture of A.~Eskin
and A.~Zorich~\cite{Eskin:Zorich}. It was first proved in
the case of the principal stratum $\cH(1,\dots,1)$ by
D.~Chen, M.~M\"oller and
D.~Zagier~\cite{Chen:Moeller:Zagier} and in the case of the
minimal stratum $\cH(2g-2)$ by A.~Sauvaget~\cite{Sauvaget}.

As a consequence for the volume asymptotics, we obtain the asymptotics
$\prop_1(\cLH)$ of the contribution of 1-cylinder square tiled surfaces
to the Masur--Veech volume of strata.
\begin{Corollary}
\label{th:contribution:1:cyl:g:to:infty}
Let  $\prop_1(\cLH)$
be  the  relative contribution of $1$-cylinder separatrix
diagrams to the volume of the stratum
$\cLH$. Then:
\begin{equation}
\label{eq:asymptotics:for:1:cyl}
\dim_\C(\cL) \cdot\prop_1(\cLH) \to 1 \text{ as }g\to+\infty\,,
\end{equation}
where the convergence is uniform for all strata in genus $g$ and where
$\dim_\C \cL =2g+\noz-1$ is the dimension of the stratum $\cLH$.
\end{Corollary}

\begin{proof}
Recall that
$\prop_1(\cLH)=\cfrac{c_1(\cLH)}{\Vol\cH(\cLH)}$. Applying
expressions~\eqref{eq:contribution:all:1:cyl:estimate}
and~\eqref{eq:asymptotic:formula:for:the:volume} for the
numerator and the denominator of the latter ratio
respectively and multiplying the result by
$d=\dim_{\C{}}\cL$ we get
$$
\zeta(d)\cdot\frac{d}{d+1}\cdot
\frac{1}{1+\epsilon(\cLH)}
\le d\cdot\prop_1(\cLH) \le
\zeta(d)\cdot\frac{d}{d-\frac{10}{29}}
\cdot
\frac{1}{1+\epsilon(\cLH)}\,,
$$
where $\epsilon(\cLH):=\epsilon(m)$ for
$\cLH=\cH(m)$.
Note that $\zeta(d)$ tends to $1$ when $d\to+\infty$.
Note also that dimensions $d$ of strata in genus $g$
vary from $2g$ to $4g-3$, so it follows from Theorem~\ref{conj:vol}
that $\epsilon(\cLH)$ tends to $0$ uniformly for all strata
$\cLH$ of dimension $d$ when $d\to+\infty$.
\end{proof}

Note that the statement in
Corollary~\ref{th:contribution:1:cyl:g:to:infty} is
equivalent to Aggarwal Theorem~\ref{conj:vol}. It would  be
very interesting to find an argument proving the
asymptotics of relative contribution of 1-cylinder
square-tiled surfaces to the Masur--Veech volume directly.

Recall  that  some  strata  are not connected. However, all the above
results  can  be easily generalized to connected components. We start
with  the  hyperelliptic  connected  components $\cH^{hyp}(2g-2)$ and
$\cH^{hyp}(g-1,g-1)$,  which  are  always very special and do not fit
the  general picture. The situation is particularly simple with them.
The results  in~\cite{AEZ:genus:0} provide a simple closed formula for the
volume of these components. These volumes are completely negligible
with respect to conjectural volume~\eqref{eq:asymptotic:formula:for:the:volume}
of the entire strata. On the other hand, each hyperelliptic component has a unique
$1$-cylinder  separatrix diagram $\cD$, which has the cyclic symmetry
group $\Aut(\cD)$ of order $d-1$ (see~Proposition~5
in~\cite{Zorich:representatives}). Thus, the contribution $c_1$ of
all $1$-cylinder diagrams is basically given by Proposition~\ref{pr:contribution:Abelian}.

\begin{Proposition}
\label{pr:proportion:hyp}
The  relative contribution of $1$-cylinder separatrix diagrams to
the volumes of the hyperelliptic components is given by the following
expressions:
\begin{align*}
\prop_1(\cH^{hyp}_1(2g-2))&=\cfrac{\zeta(2g)}{\pi^{2g}}\cdot
2g(2g+1)\cdot
\cfrac{(2g-2)!!}{(2g-3)!!} \sim
4\cdot\frac{g^{5/2}}{\pi^{2g-1/2}}\,.
\\
\prop_1(\cH^{hyp}_1(g-1,g-1))&=\cfrac{\zeta(2g+1)}{2\pi^{2g}}\cdot
(2g+1)(2g+2)\cdot
\cfrac{(2g-1)!!}{(2g-2)!!} \sim
4\cdot\frac{g^{5/2}}{\pi^{2g+1/2}}\,.
\end{align*}
\end{Proposition}

Proposition~\ref{pr:proportion:hyp} shows that the resulting relative
contribution  $\prop_1$  of $1$-cylinder separatrix diagrams to the volumes
of the hyperelliptic components is completely negligible with respect
to~\eqref{eq:asymptotics:for:1:cyl}.  It  is  proved  in  the  end of
section~\ref{ss:1:cylinder:diagrams:Abelian}.

It     remains     to     consider     nonhyperelliptic    components
$\cH^{even}(2m_1,\dots,2m_\noz)$ and $\cH^{odd}(2m_1,\dots,2m_\noz)$.
Recall another conjecture from \cite{Eskin:Zorich}:

\begin{Conjecture}[{\cite[Conjecture 2]{Eskin:Zorich}}]
\label{conj:even:odd}
The ratio of volumes
of  even and odd components of strata $\cH(2m_1,\dots,2m_\noz)$ tends
to  $1$  uniformly for all partitions $m_1+\dots+m_\noz=g-1$ as genus
$g$ tends to infinity, i.~e.
$$
\lim_{g\to+\infty}\frac
{\Vol\cH^{even}(2m_1,\dots,2m_\noz)}
{\Vol\cH^{odd}(2m_1,\dots,2m_\noz)}
=1
$$
uniformly in $m_1,\dots , m_\noz$.
\end{Conjecture}

By the result in~\cite[Theorem 4.19]{Delecroix}, the ratio of the weighted numbers
of $1$-cylinder separatrix diagrams in the connected components
$\cH^{even}_1(2m_1,\dots,2m_\noz)$ and $\cH^{odd}_1(2m_1,\dots,2m_\noz)$ also tends to $1$
uniformly for all partitions $m_1+\dots+m_\noz=g-1$ as genus
$g$ tends to infinity. Thus,  we  obtain the following statement.

\begin{condCorollary}
\label{cor:even:odd:1:cyl}
Conjecture~\ref{conj:even:odd} and Conjecture~\ref{conj:vol}
restricted  to  the  strata with zeroes of even  degrees  are
together  equivalent  to  the  following  statement: for any
partition  $(m_1,\dots,m_\noz)$  of $g-1$ into a sum of strictly
positive integers $m_1+\dots+m_\noz=g-1$ one has
\begin{align*}
d\cdot\prop_1(\cH^{even}(2m_1,\dots,2m_\noz)) &\to 1 \text{ as }g\to+\infty\\
d\cdot\prop_1(\cH^{odd}(2m_1,\dots,2m_\noz)) &\to 1 \text{ as }g\to+\infty\,,
\end{align*}
where  $d=2m_1+\dots+2m_\noz+\noz+1$  and  convergence is uniform for
all strata in genus $g$.
\end{condCorollary}

Since we do not want to overload the current paper, the
questions concerning the asymptotic proportions
$\prop_k(\cLH)$ of $k$-cylinder diagrams for
$k=2,3,\dots$ for strata of high genera will be
addressed in a separate paper. In this forthcoming paper we
will treat, in particular, the question of the dependence
of $\prop_k(\cLH)$ on the genus and the dimension of the
stratum, and the question of the limit distribution of
$\prop_k(\cLH)$ with respect to all possible $k$ for
strata of large genera.

\subsection{Application: experimental evaluation of the Masur--Veech volumes}
\label{ss:Application:experimental:evaluation:of:MV:volumes}
Let $\cL$ be a component of a stratum of
Abelian differentials or of meromorphic quadratic
differentials with at most simple poles. We first present
a Monte-Carlo method\footnote{The term \textit{Monte-Carlo}
refers to the fact that the output of our algorithm is a
\textit{random} approximation of the volume. The quality of
approximation depends on the randomly chosen sample of
integer points.} to approximate $\prop_1(\cL)$
via Theorem~\ref{th:iet} (Abelian case) or
Theorem~\ref{th:generalized:iet} (quadratic case).
Pick a permutation or generalized permutation $\pi$ whose suspensions
belong to $\cL$. Take a relatively compact box $V$ in $\R_+^d$ or $C^d_+(\pi)$.
Then fix a large number $N$ and for a sample of lengths $\lambda$
in $V \cap \frac{1}{N} \N$ compute the proportion of one cylinder
interval exchanges among the $(\pi, \lambda)$. This gives an
approximation of the relative contribution $\prop_1(\cL)$ of
$1$-cylinder diagrams to the volume of the chosen component
of stratum $\cL$.

Now, one can perform an exact count of the weighted
number of $1$-cylinder separatrix diagrams (where the
weight is reciprocal to the order of the symmetry group
of the diagram). Applying
Proposition~\ref{pr:contribution:Abelian} (respectively,
Proposition~\ref{pr:contribution:quadratic}) we obtain
the exact value $c_1(\cL)$ of the contribution of
$1$-cylinder diagrams to the volume. Since, we already
know approximately, what part of the total value makes the
resulting volume, we obtain an approximate value of the
volume of the ambient stratum. The experimental and
theoretical values of the volumes of low dimensional strata
of quadratic differentials are compared in
Appendix~C in the initial longer arXiv
version~\cite{DGZZ:equidistribution} of the current paper.

\subsection{Contribution of a single $1$-cylinder separatrix diagram: computation}
\label{ss:contribution:of:one:1:cylinder:diagram:computation}

Consider  Jenkins--Strebel  differentials
represented  by  a  \emph{single}  flat cylinder $C$ filled by closed
horizontal leaves. Note that all zeroes and poles (critical points of
the  horizontal  foliation)  of  such differential are located on the
boundary of this cylinder.

\begin{figure}[htb]
\includegraphics{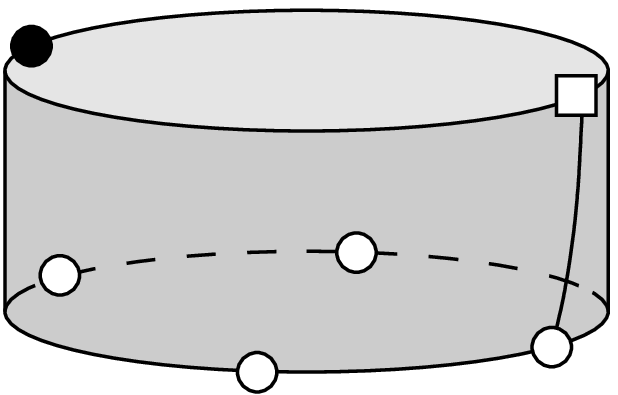}
\includegraphics{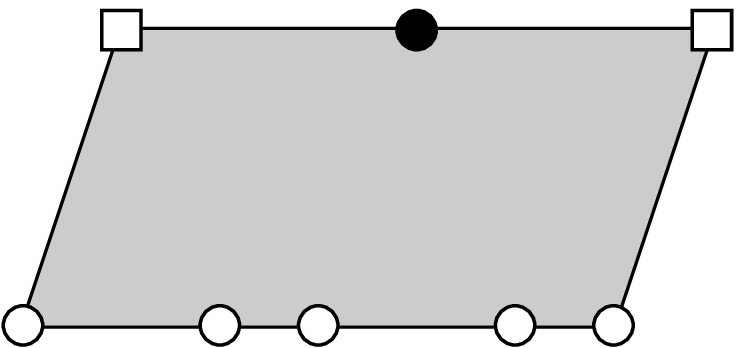}
\includegraphics{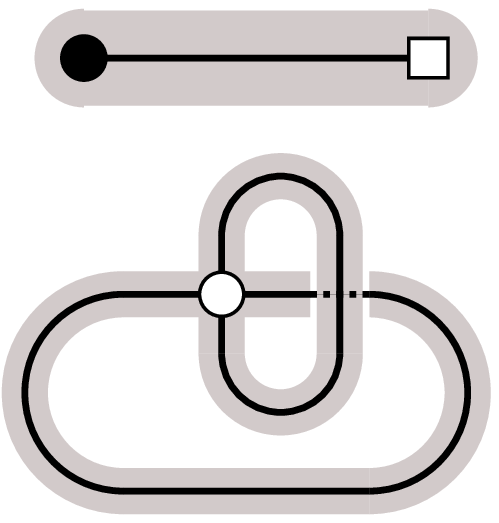}
\begin{picture}(0,0)(160,5)
\put(40,7){$X_1$}
\put(25,-10){$X_1$}
\put(17,-29){$X_3$}
\put(63,-27){$X_0$}
\put(0,-54){$X_2$}
\put(48,-56){$X_3$}
\put(76,-51){$X_2$}
\end{picture}
\begin{picture}(0,0)(40,-105)
\put(25,-110){$X_1$}
\put(69,-110){$X_1$}
\put(2,-168){$X_2$}
\put(24,-168){$X_3$}
\put(46,-168){$X_2$}
\put(67,-168){$X_3$}
\put(-16,-140){$X_0$}
\put(87,-140){$X_0$}
\end{picture}
\begin{picture}(0,0)(-73,-100)
\put(38,-100){$X_1$}
\put(2,-138){$X_2$}
\put(28,-155){$X_3$}
\end{picture}
\vspace{60bp}
\caption{
\label{fig:Jenkins:Strebel}
A  Jenkins--Strebel  differential  with a single cylinder,
one of its parallelogram  patterns, and its ribbon graph representation.
We have $l=0, m=1, n=2$. The stratum is $\cQ(2,-1^2)$.
}
\end{figure}

Each of the two boundary components $\partial C^+$ and $\partial C^-$
of the cylinder is subdivided into a collection of horizontal saddle
connections $\partial C^+=X_{\alpha_1}\sqcup\dots\sqcup X_{\alpha_r}$
and $\partial C^-=X_{\alpha_{r+1}}\sqcup\dots\sqcup X_{\alpha_s}$.
The subintervals are naturally organized in pairs of
equal length; subintervals in every pair are identified by a natural
isometry which preserves the orientation of the surface. Denoting
both subintervals in the pair representing the same saddle connection
by the same symbol, we encode the combinatorics of identification of
the boundaries of the cylinder by two lines of symbols,

\begin{equation}
\label{eq:js:prepermutation}
\begin{picture}(0,0)(-2,0)
\put(-3,10){\vector(1,0){0}}
\put(-5,15){\oval(10,10)[bl]}
\put(30,15){\oval(80,10)[t]}
\put(65,15){\oval(10,10)[br]}
\put(-3,-4){\vector(1,0){0}}
\put(-5,-9){\oval(10,10)[tl]}
\put(40,-9){\oval(100,10)[b]}
\put(85,-9){\oval(10,10)[tr]}
\end{picture}
\begin{matrix}\alpha_1&\dots&\alpha_r&\\ \alpha_{r+1}\!\!&\dots\!&\!\dots\!&\!\alpha_s\end{matrix}
\vspace{8pt}
\end{equation}

\noindent
where the symbols in each line are organized in a cyclic order.
\medskip

\noindent\textbf{Choice of cyclic ordering.}
There  are  two  alternative conventions on the choice of this cyclic
order.  Note  that our surface is oriented (and not only orientable).
Hence,  this  orientation  induces  a  natural orientation of each of
$\partial  C^+$ and of $\partial C^-$ which defines a cyclic order on
the symbols labeling the segments.

Note  also  that  if  we have an Abelian differential, its horizontal
foliation  is  oriented.  The  corresponding  orientation  of  leaves
defines  the  same  cyclic  order as the previous one on one boundary
component  of the cylinder and the opposite cyclic order on the other
boundary component of the cylinder.

For  quadratic differentials the foliation is nonorientable. However,
for  a  Jenkins--Strebel  differential  we  can coherently choose the
orientation  of  all  regular  leaves in the interior of each maximal
cylinder,  and  it  induces  the  cyclic order of symbols labeling the
segments on $\partial C^+$ and $\partial C^-$. Similarly to the case of
Abelian  differentials, this  cyclic  ordering coincides with the one
induced  by  the  orientation  of  the  surface  on  one  of  the two
components  $\partial C^+, \partial C^-$ and provides the opposite cyclic
ordering on the other component.

In~\eqref{eq:js:prepermutation}
we  use  the  cyclic  ordering  coming  from  the orientation of the
foliation and not from the orientation of the surface.
\medskip

\noindent\textbf{Abelian versus quadratic differentials.}
By   construction,   every   symbol  appears  exactly  twice  in  two
lines~\eqref{eq:js:prepermutation}.  If  all the symbols in each line
are  distinct, the resulting flat surface has trivial linear holonomy
and  corresponds  to  an  Abelian  differential.  In  this case every
interval  on one side of the cylinder is identified with and interval
on  the  other  side  and vice versa, so there are no relations
between  the lengths of the intervals. In other words, any orientable
separatrix diagram having only two boundary components is realizable.

Otherwise, a flat metric of the resulting closed surface has holonomy
group  $\ZZ$;  in  the  latter  case  it corresponds to a meromorphic
quadratic  differential with at most simple poles. In this case there
is a linear relation between the lengths of the intervals: the sum of
lengths  of all intervals on one side of the cylinder is equal to the
sum  of  lengths of all intervals on the other side. This implies the
following  combinatorial  restriction: the set of symbols in one line
cannot  be a proper subset of the set of symbols on the complimentary
line.  This  condition  is  a  necessary  and sufficient condition of
realizability  for  a non-orientable separatrix diagram. For example,
the following combinatorial data

\begin{equation*}
\begin{picture}(0,0)(-2,0)
\put(7,10){\vector(1,0){0}}
\put(5,15){\oval(10,10)[bl]}
\put(25,15){\oval(50,10)[t]}
\put(45,15){\oval(10,10)[br]}
\put(7,-4){\vector(1,0){0}}
\put(5,-9){\oval(10,10)[tl]}
\put(40,-9){\oval(80,10)[b]}
\put(75,-9){\oval(10,10)[tr]}
\end{picture}
\begin{matrix}&1&2&3&\\ &3&4&1&2&4\end{matrix}
\vspace{8pt}
\end{equation*}
do  not  admit  any  strictly  positive  solution  for the lengths of
subintervals, while

\begin{equation*}
\begin{picture}(0,0)(-2,0)
\put(-3,10){\vector(1,0){0}}
\put(-5,15){\oval(10,10)[bl]}
\put(30,15){\oval(80,10)[t]}
\put(65,15){\oval(10,10)[br]}
\put(-3,-4){\vector(1,0){0}}
\put(-5,-9){\oval(10,10)[tl]}
\put(30,-9){\oval(80,10)[b]}
\put(65,-9){\oval(10,10)[tr]}
\end{picture}
\begin{matrix}5&1&2&3&\!5\\ 3&4&1&2&\!4\end{matrix}
\vspace{8pt}
\end{equation*}
admits   strictly   positive   solutions   satisfying   the  relation
$\lambda_4=\lambda_5$.
\bigskip

\noindent\textbf{Contribution of each individual $1$-cylinder separatrix diagram.}
Now everthing is ready for the proofs of
Propositions~\ref{pr:contribution:Abelian}
and~\ref{pr:contribution:quadratic}.

\begin{proof}[Proof of Proposition~\ref{pr:contribution:Abelian}]
An  orientable  $1$-cylinder  separatrix diagram $\cD$
representing a stratum  of  Abelian differentials of
complex dimension $d$ has $d-1$ separatrices (horizontal
saddle connections). Denote the length of the $i$-th
separatrix by $\lambda_i$. The  perimeter  $w$  of  the
cylinder  is  equal  to  the sum of the lengths          of
all         separatrices,         namely
$w=\lambda_1+\lambda_2+\dots+\lambda_{d-1}$. Denote  by $h$
the height of the cylinder.   Finally,   denote   by
$\phi$   the   ``twist'',  where $0\le\phi<w$.  The  number
of  square-tiled  surfaces  tiled with at most  $N$  unit
squares  and  having $\cD$ as the separatrix diagram equals
\begin{multline*}
\cfrac{1}{|\Aut(\cD)|}\ \sum_{\substack{\lambda_1,\dots,\lambda_{d-1},h\in\N\\
w=\lambda_1+\dots+\lambda_{d-1}\\
w\cdot h\le N}}
w \ \approx \
\cfrac{1}{|\Aut(\cD)|}\
\sum_{\substack{w,h\in\N\\w\cdot h\le N}}
w\cdot\cfrac{w^{d-2}}{(d-2)!}
\,
=\\=\,
\cfrac{1}{|\Aut(\cD)|}\
\cfrac{1}{(d-2)!}\ \sum_{\substack{w,h\in\N\\w\le \frac{N}{h} }} w^{d-1}
\, \approx\,
\cfrac{1}{|\Aut(\cD)|}\
\cfrac{1}{(d-2)!}\, \sum_{h\in\N} \cfrac{1}{d}\,\cdot
\left(\cfrac{N}{h}\right)^d
\,=\\= \,
\cfrac{1}{|\Aut(\cD)|}\
\cfrac{N^d}{(d-2)!}\,
\cfrac{1}{d}\,
\cdot\sum_{h\in\N} \cfrac{1}{h^d}
\, = \,
\cfrac{1}{|\Aut(\cD)|}\
\cfrac{1}{d}\cdot
\cfrac{N^d}{(d-2)!}\,\cdot
\zeta(d)\,.
\end{multline*}
The above expression gives the asymptotic
number of square-tiled surfaces corresponding to the
diagram $\cD$ tiled with at most $N$ unit squares.
By equation~\eqref{eq:volume:as:limit}  the contribution of any
such  term  to  the  Masur--Veech volume
$\Vol\cLH^{\mathit{unnumbered}}$
of the stratum
with \textit{unnumbered} zeroes is computed by multimplying by
$\cfrac{2d}{N^d}$ and by passing to the limit when $N\to+\infty$.
Thus, the contribution of the
$1$-cylinder separatrix diagram $\cD$ to the volume of the ambiant
stratum with \textit{unnumbered} zeroes
is
\begin{equation*}
\label{eq:contribution:not:numbered}
\cfrac{1}{|\Aut(\cD)|}\cdot
\cfrac{2}{(d-2)!}\,\cdot \zeta(d)\,.
\end{equation*}

Representing the set $\{m_1,\dots,m_\noz\}$ as
$\{1^{\mult_1},2^{\mult_2},\dots\}$ we get the following formula for
the contribution of an individual rooted diagram to the
Masur--Veech volume
$\Vol\cH(m_1,\dots,m_\noz)$ of the stratum with
\textit{numbered} zeroes:
$$
\cfrac{2}{|\Aut(\cD)|}\cdot
\cfrac{\mult_1!\cdot \mult_2! \cdots}{(d-2)!}\,\cdot \zeta(d)\,.
$$
which completes the proof of Proposition~\ref{pr:contribution:Abelian}.
\end{proof}

\begin{proof}[Proof of Proposition~\ref{pr:contribution:quadratic}]
The  evaluation of the contribution of an $1$-cylinder diagram to the
volume of a stratum of quadratic differentials is analogous. The only
difference  is  that  it  gets  an  extra  weight  depending  on  the
additional discrete parameters $l,m,n$ of the diagram.

Consider a nonorientable $1$-cylinder separatrix diagram. Each
separatrix (i.e. each horizontal saddle  connection) is represented
by two intervals on  the boundary of the cylinder. One may have one
interval on each of the two boundary  components,  both intervals on
the ``top'' boundary component of the cylinder, or both on the
``bottom'' boundary component. Recall that we denote the number   of
corresponding   saddle connections  by  $l,m,n$  correspondingly.

We start with a more general situation when $l>0$. Introduce
the following notation:
\begin{align*}
w_1&:=\lambda_{i_1}+\dots+\lambda_{i_l}\\
w_2&:=2(\lambda_{j_1}+\dots+\lambda_{j_m})=2(\lambda_{k_1}+\dots+\lambda_{k_n})\,,
\end{align*}
where  by  $\lambda_{i_s}$,  $s=1,\dots,l$  we denote the lengths of the
segments  which  are  present on the both  sides  of the cylinder, by
$\lambda_{j_s}$,  $s=1,\dots,m$  we  denote  the lengths of the segments
which are present only on top of the cylinder, and by $\lambda_{k_s}$,
$k=1,\dots,n$  we  denote  the  lengths  of  the  segments  which are
present  only  on  the  bottom  of  the  cylinder.  For example, on
Figure~\ref{fig:Jenkins:Strebel} the segment  $X_1$ is present only on the
top,  the segments  $X_2,  X_3$  --- only on the bottom, and there are no
other  segments,  so  we have $l=0, m=1, n=2$.

In this notation the
length $w$ of the waist curve (perimeter) of the cylinder is equal to
$w=w_1+w_2$.  When  $l>0$  (that  is  when the boundary components of the
cylinder share at least one common interval) the waist curve $\gamma$
of  the  cylinder is not homologous to zero. Under our assumptions on
the normalization (see Convention~\ref{conv:lattice} for details)
the   lengths   $\lambda_s$  of  all  subintervals  are
half-integers,  $w_1$ is a half-integer, $w_2$ is automatically an integer,
and $w$ is a half-integer.

The number of compositions of an integer $n$ into exactly
$k$ parts is given by the binomial coefficient
$\binom{n-1}{k-1}$. Thus, the  leading term in the number
of ways to represent $w_1\gg l$ as a sum of $l$
half-integers
$$
w_1=\lambda_{i_1}+\dots+\lambda_{i_l}
$$
is
$$
2^{l-1}\cfrac{w_1^{l-1}}{(l-1)!}\,.
$$
The  leading term in the number of ways to represent $w_2$ as a sum
of $m$ (respectively $n$) integers
$$
w_2=2\lambda_{j_1}+\dots+2\lambda_{j_m}=2\lambda_{k_1}+\dots+2\lambda_{k_n}
$$
is
$$
\cfrac{w_2^{m-1}}{(m-1)!}
\qquad\left(\text{respectively }\
\cfrac{w_2^{n-1}}{(n-1)!}
\right)
\,.
$$

Denote by $h$ the half-integer height of our single cylinder and
introduce the integer parameter $H=2h$. The condition $w\cdot h\le
N/2$ on the area of the surface translates as $w\cdot H\le N$ in
terms of the parameter $H$. Thus, introducing the notation $W:=2w$,
we can represent the leading term in the corresponding sum as
\begin{multline*}
\sum_{\substack{w\in\frac{1}{2}\N\\H\in\N\\
w\cdot H\le N}}\
\sum_{\substack{w_2\in\N\\ w_2< w}}
2w\cdot 2^{l-1}\cfrac{(w-w_2)^{l-1}}{(l-1)!}
\cdot\cfrac{w_2^{m-1}}{(m-1)!}\cdot\cfrac{w_2^{n-1}}{(n-1)!}
=\\=
\frac{2^{l-1}}{(l-1)!(m-1)!(n-1)!}
\sum_{\substack{W, H\in\N\\W\cdot H\le 2N}} W
\sum_{w_2=1}^{\lfloor W/2\rfloor}(W/2-w_2)^{l-1}w_2^{m+n-2}
\sim \\ \sim
\frac{2^{l-1}}{(l-1)!(m-1)!(n-1)!}
 \cdot
 \sum_{H\in \N}\sum_{W=1}^{\lfloor{2N/H}\rfloor}
W\cdot
\left(\frac{W}{2}\right)^{l+m+n-2}
\cdot \int_0^1(1-u)^{l-1}u^{m+n-2}\,du
 \\ \sim
\frac{2^{l-1}}{(l-1)!(m-1)!(n-1)!}
\cdot \frac{(l-1)!(m+n-2)!}{(l+m+n-2)!}
\cdot\\
 \cdot \frac{1}{2^{l+m+n-2}} \cdot
\sum_{H\in\N} \cfrac{1}{l+m+n}\cdot\left(\cfrac{2N}{H}\right)^{l+m+n}
\\ =
\frac{2^{l+1}(m+n-2)!}{(m-1)!(n-1)!(l+m+n-2)!}\cdot
\cfrac{N^{l+m+n}}{l+m+n}\cdot\zeta(l+m+n)\,.
\end{multline*}
where we used the relation
$$
\int_0^1 u^a (1-u)^b\, du = \cfrac{a!\, b!}{(a+b+1)!}\,.
$$
The above expression gives the asymptotic
number of square-tiled surfaces corresponding to the
diagram $\cD$ tiled with at most $2N$ squares of the size
$\frac{1}{2}\times\frac{1}{2}$. By
equation~\eqref{eq:volume:as:limit:quadratic}  the
contribution of any such  term  to  the  Masur--Veech
volume $\Vol\cLQ^{\mathit{unnumbered}}$ of the stratum with
\textit{unnumbered}  zeroes is  computed  by multimplying
the latter expression by $\cfrac{2d}{N^d}$ and by passing
to the limit when $N\to+\infty$.
It remains to note that $d=l+m+n$ to obtain
the contribution of $\cD$
to the volume of the corresponding
stratum with \textit{anonymous} (\textit{non-numbered})
zeroes and poles:
\begin{equation*}
\frac{2^{l+2}(m+n-2)!}{(m-1)!(n-1)!(l+m+n-2)!}
\cdot\zeta(l+m+n)
\end{equation*}

Multiplying  the  result by the product of factorials responsible for
numbering    the    zeroes    and   poles,   we   get   the   desired
formula~\eqref{eq:general:contribution}.

In  the  remaining  particular  case  when  $l=0$  (that  is, when the
boundary  components  of  the  cylinder  do not share a single common
saddle  connection)  the  waist  curve  $\gamma$  of  the cylinder is
homologous  to zero, while $\hat\gamma$ is not. Under our assumptions
on the normalization,  the  lengths  $\lambda_s$  of  all  subintervals are
half-integers,  and  $w=w_2$  is  automatically an integer, as it should
be. Performing a completely analogous computation we get a particular
case of formula~\eqref{eq:general:contribution} where $l=0$.
\end{proof}

\subsection{Counting $1$-cylinder diagrams for strata of Abelian
differentials based on Frobenius formula and Zagier bounds}
\label{ss:1:cylinder:diagrams:Abelian}

Enumeration  of  orientable  $1$-cylinder  separatrix
diagrams through Frobenius formula was elaborated
in~\cite{Delecroix}.  Consider  some stratum of Abelian differentials
$\cH(m_1,\dots,m_\noz)$. Let
\begin{equation}
\label{eq:n}
n=\sum_{i=1}^r  (m_i+1)=2g-2+r=\dim_\C\cH(m_1,\dots,m_\noz)-1\,.
\end{equation}
Denote  by  $C(\psi)$  the conjugacy class of a permutation $\psi$ in
the  symmetric  group  $S_n$;  denote  by  $C(\sigma)$ the
conjugacy  class  of the cyclic permutation $\sigma=(1,2,\dots,n)$ in
$S_n$. Finally, denote by $C(\nu)$  the  conjugacy  class
of the product of $r$ cycles of lengths $(m_1+1,\dots,m_\noz+1)$.

Following~\cite{Zagier}                   denote                   by
$\cN(S_n;C(\sigma),C(\sigma),C(\nu))$    the    number    of
solutions  of the equation $c_1 c_2 c_3 =1$, where the permutations $c_1$
and   $c_2$  belong  to  the  conjugacy  class  $C(\sigma)$  and  the
permutation $c_3$ belongs to the conjugacy class $C(\nu)$:
\begin{multline}
\label{eq:N:C:C:C}
\cN(S_n;C(\sigma),C(\sigma),C(\nu))
=\\=
\#\{(c_1,c_2,c_3)\in C(\sigma)\times C(\sigma)\times C(\nu)\,|\,
c_1 c_2 c_3 =1\}\,.
\end{multline}

Every  such  solution  defines  a    $1$-cylinder  separatrix diagram
corresponding to the stratum $\cH(m_1,\dots,m_\noz)$. Indeed, consider a
horizontal  cylinder $S^1\times [0;1]$ such that each of its boundary
components  is  subdivided into $n$ segments. Choose the orientation of
the  boundary  components  induced by the orientation of the circle
$S^1$  (on  one of the two components it differs from the orientation
induced  from the orientation on the cylinder) and assign labels from
$1$  to $n$ to the subintervals of one boundary component in such a way
that  they  appear  in  the cyclic order $c_1$, and assign labels to the
remaining boundary component in such a way that they appear in the cyclic
order  $c^{-1}_2$.  Cut the cylinder along the horizontal waist curve
and  identify  pairs  of  subintervals  on  the  boundary  components
carrying  the  same  labels  respecting  the orientation induced from
$S^1$. Consider the $1$-cylinder separatrix diagram $\cD$ represented
by   the  resulting  ribbon  graph.  The  relation  $c_1\cdot  c_2  =
c^{-1}_3$,  where  $c_3\in C(\nu)$, guarantees that $\cD$ corresponds
to the stratum $\cH(m_1,\dots,m_\noz)$.

\begin{Example}(See~\cite{Zorich:representatives} for details.)
Consider the pair of cyclic permutations
$c_1=(1,2,3,4,5,6,7,8)$ and
$c_2=(4,3,2,5,8,7,6,1)$
in $S_8$. The two boundary components of the corresponding
horizontal cylinder get the following labeling:

\begin{equation}
\label{eq:H1111:cyclic}
\begin{picture}(-4,6)
\put(-3,10){\vector(1,0){0}}
\put(-5,15){\oval(10,10)[bl]}
\put(70,15){\oval(160,10)[t]}
\put(145,15){\oval(10,10)[br]}
\put(-3,-4){\vector(1,0){0}}
\put(-5,-9){\oval(10,10)[tl]}
\put(70,-9){\oval(160,10)[b]}
\put(145,-9){\oval(10,10)[tr]}
\end{picture}
\begin{matrix}
1\to 2\to 3\to 4\to 5\to 6\to 7\to 8\\
4\to 3\to 2\to 5\to 8\to 7\to 6\to 1
\end{matrix}
\vspace{8pt}
\end{equation}

The    corresponding    translation   surface   is   represented   in
Figure~\ref{fig:merging:zeroes}  in two different ways: as a cylinder
(rather a parallelogram) with pairs of corresponding sides identified
by  parallel translations and as a ribbon graph (separatrix diagram).
The  core  of  the  corresponding  ribbon  graph has four vertices of
valence  four  representing  four  conical  singularities  of  angles
$4\pi$, or, equivalently, four simple zeroes of the resulting Abelian
differential.  Each  edge of the ribbon graph represents a horizontal
saddle   connection   (separatrix).   Turning   around  zeroes  in  a
counterclockwise  direction,  see Figure~\ref{fig:merging:zeroes}, we
see  the  incoming  horizontal  separatrix  rays appear in the cyclic
orders  given  by  the  cyclic  decomposition of $c_1\cdot c_2^{-1}$,
namely
$$
c_1\cdot c_2^{-1}=(1,3)(2,4)(5,7)(6,8)
$$

\begin{figure}[htb]
\includegraphics{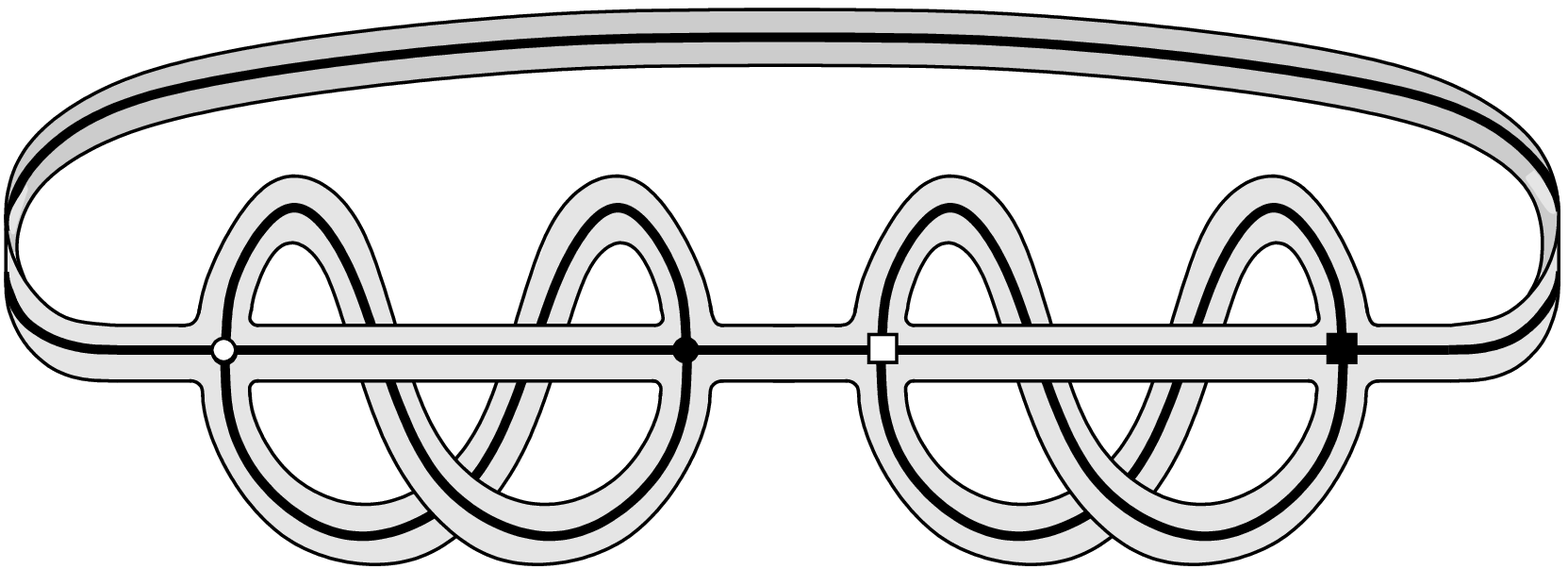}
\includegraphics{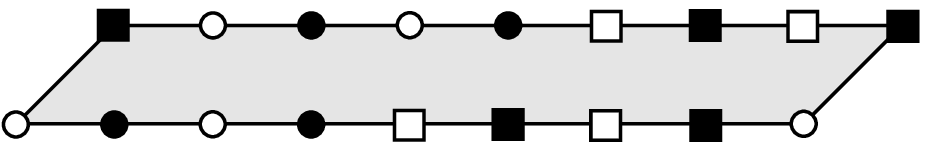}
\begin{picture}(0,0)(35,-89)
\put(-88,-138){$X_1$}
\put(-59,-112){$X_2$}
\put(-32,-138){$X_3$}
\put(-1,-112){$X_4$}
\put(30,-138){$X_5$}
\put(61,-112){$X_6$}
\put(90,-138){$X_7$}
\put(120,-112){$X_8$}
\put(146,-138){$X_1$}
\end{picture}

\begin{picture}(0,0)(3,-13)
\begin{picture}(0,0)(0,5)
\put(-88,-112){1}
\put(-60,-112){2}
\put(-32,-112){3}
\put(-4,-112){4}
\put(24,-112){5}
\put(53,-112){6}
\put(82,-112){7}
\put(110,-112){8}
\end{picture}
\begin{picture}(0,0)(28,5)
\put(-89,-155){4}
\put(-60,-155){3}
\put(-32,-155){2}
\put(-4,-155){5}
\put(24,-155){8}
\put(53,-155){7}
\put(82,-155){6}
\put(110,-155){1}
\end{picture}
\begin{picture}(0,0)(3,5)
\put(-122,-131){0}
\put(115,-136){0}
\end{picture}
\end{picture}
\vspace{150bp}
\caption{
\label{fig:merging:zeroes}
The  ribbon  graph  representation of a Jenkins--Strebel differential
with   a   single   cylinder   (top   picture)  versus  the  cylinder
representation  (bottom  picture).  All vertices marked with the same
symbols are identified to a single conical singularity. }
\end{figure}
\end{Example}

It  is  clear  that  a simultaneous conjugation of permutations $c_1,
c_2,  c_3$  by  the same permutation does not change the $1$-cylinder
diagram.  In  particular, we can choose $c_1=\sigma$. Note also, that
our  diagrams  do  not  have any distinguished (marked) intervals. We
have $|C(\sigma)|=(n-1)!$ for cardinality of $C(\sigma)$, and we have
$n$  ways  to  attribute  index  $1$  to  one of the intervals at the
bottom.    Thus,    we    have    proved    the    following    Lemma
from~\cite{Delecroix}:

\begin{Lemma}
The  weighted  number $\cN_1(m_1,\dots,m_\noz)$ of $1$-cylinder diagrams
$\cD$  for  a given stratum $\cH(m_1,\dots,m_\noz)$, where the weight is
the  inverse of the order of the group of symmetries, is expressed as
\begin{equation}
\label{eq:N:m1:mr}
\cN_1(m_1,\dots,m_\noz)
=\sum_{\substack{\text{One-cylinder}\\ \text{diagrams }\cD\\
\text{in the stratum}\\ \cH(m_1,\dots,m_\noz)}}
\frac{1}{|\Aut(\cD)|}
=
\cfrac{1}{n!}\cdot
\cN(S_n;C(\sigma),C(\sigma),C(\nu))
\end{equation}
\end{Lemma}

Now we are ready to prove Theorem~\ref{th:contribution:all:1:cyl:estimate}.

\begin{proof}[Proof of Theorem~\ref{th:contribution:all:1:cyl:estimate}]
Following~\cite{Zagier:bounds} denote by $R(\psi)$ the number of
ways to represent an even permutation $\psi$ in $S_n$
as a product of two $n$-cycles. Clearly,
\begin{equation}
\label{eq:N:through:R}
\cN(S_n;C(\sigma),C(\sigma),C(\psi)=R(\pi)\cdot|C(\psi)|\,.
\end{equation}
From now on choose any $\psi\in C(\nu)$, where $C(\nu)$ is
the conjugacy class of the
product of $r$ cycles of lengths $m_1+1,\dots,m_\noz+1$
respectively. The cardinality of $C(\psi)$ is given by
\begin{equation}
\label{eq:card:C:m}
|C(\psi)|=|C(\nu)|=n!\cdot\prod_k \frac{1}{\mult_k! (k+1)^{\mu_k}}\,,
\end{equation}
where  $\mult_k$  is  the  multiplicity of the entry $k=1,2,\dots$ in
$(m_1,\dots,m_\noz)$.

Denote  by  $c(m_1,\dots,m_\noz)$  the  absolute
contribution of all $1$-cylinder
diagrams    to    the   volume   $\Vol\cH(m_1,\dots,m_\noz)$   as   in
equation~\eqref{eq:contribution:all:1:cyl:estimate}              from
Theorem~\ref{th:contribution:all:1:cyl:estimate}.     Recall     that
$d=\dim\cH(m_1,\dots,m_\noz)=n+1$.

Nesting~\eqref{eq:card:C:m}                 in~\eqref{eq:N:through:R}
in~\eqref{eq:N:m1:mr}      and      combining     it     with     the
formula~\eqref{eq:contribution:numbered}                         from
Proposition~\ref{pr:contribution:Abelian}  for the contribution of an
individual    $1$-cylinder    diagram    to   the   volume   we   get
\begin{multline*}
c(m_1,\dots,m_\noz)=
\frac{1}{n!}
\cdot
\left(n!\cdot\prod_k \frac{1}{\mult_k! (k+1)^{\mu_k}}\right)
\cdot R(\psi)
\cdot
\cfrac{\mult_1!\cdot \mult_2! \cdots}{(n-1)!}\,\cdot 2\zeta(n+1)
=\\=
\frac{R(\psi)}{(n-1)!}\cdot
\frac{2\zeta(n+1)}{(m_1+1)\cdot\dots\cdot(m_\noz+1)}\,.
\end{multline*}
By Theorem~2 in~\cite{Zagier:bounds} the following universal
bounds are valid:
$$
\frac{2(n-1)!}{n+2}\le R(\psi) \le
\frac{2(n-1)!}{n+\frac{19}{29}}\,.
$$
Plugging these bounds in the latter expression for $c(m_1,\dots,m_\noz)$
in  terms  of  $R(\psi)$  and returning to notation $d=n+1$ we obtain the
bounds~\eqref{eq:contribution:all:1:cyl:estimate}                from
Theorem~\ref{th:contribution:all:1:cyl:estimate}.
\end{proof}
\medskip

\noindent\textbf{Frobenius formula.}
We now apply the Frobenius formula to prove
Theorem~\ref{th:contribution:all:1:cyl}    and   then   we
evaluate explicitly  the  contribution  of  all
$1$-cylinder  diagrams to the volume of the ambient stratum
for the minimal stratum $\cH(2g-2)$ and for   the
principal   stratum   $\cH(1,\dots,1)$,  and  thus  prove
Corollary~\ref{cor:total:contribution:min:and:principal}.
Note  that for   $g>3$   the   stratum   $\cH(2g-2)$
contains  three  connected components.  Contribution  of
all $1$-cylinder diagrams to individual components is
described in Proposition~\ref{pr:proportion:hyp} and in the
Conditional Corollary~\ref{cor:even:odd:1:cyl}.

\begin{proof}[Proof of Theorem~\ref{th:contribution:all:1:cyl}]
Applying Frobenius formula in the notation of~(A.8) in~\cite{Zagier}, we
express the quantity~\eqref{eq:N:C:C:C}  as a sum over characters $\chi$
of the symmetric group $S_n$:
\begin{multline}
\label{eq:Frobenius:Formula:1}
\cN(S_n;C(\sigma),C(\sigma),C(\nu))
=\\=
\cfrac{|C(\sigma)|\cdot|C(\sigma)|\cdot|C(\nu)|}{|S_n|}
\,
\sum_\chi \frac{\chi(C(\sigma))\chi(C(\sigma))\chi(C(\nu))}{\chi(1)^{3-2}}\,.
\end{multline}
In  our particular case the cardinality of the conjugacy class of the
long cycle $\sigma$ is $|C(\sigma)|=(n-1)!$ and $|S_n|=n!$.

Following  the  notation  of   \S A.2   in~\cite{Zagier},   denote   by
$\mathbf{St}_n=\C^n/\C$  the  standard  irreducible representation of
dimension $n-1$ of the group $S_n$ and put
$$
\chi_j(g):=\operatorname{tr}(g,\pi_j)\qquad
\pi_j:=\wedge^j(\mathbf{St}_n)\qquad
(0\le j\le n-1)\;,
$$
where $g\in S_n$ is any permutation.
It  is  known  that  the  representations $\pi_j$ are irreducible and
pairwise distinct for $0\le j\le n-1$ (Lemma~A.2.1 in~\cite{Zagier}).
Moreover, by Lemma~A.2.2 in~\cite{Zagier} for any irreducible
representation $\pi$ one has
$$
\chi_\pi(\sigma)=
\begin{cases}
(-1)^j,&\text{if }\pi\simeq\pi_j\text{ for some }j,\ 0\le j\le n-1\\
0&\text{otherwise\,,}
\end{cases}
$$
where  $\sigma=(1,2,\dots,n)$  is the maximal cycle in $S_n$.

Finally, $\chi_j(1)=\dim\pi_j=\binom{n-1}{j}$.

Substituting   all   these   values  in  the  Frobenius  formula  we  can
rewrite~\eqref{eq:Frobenius:Formula:1} as
\begin{multline}
\label{eq:Frobenius:Formula:2}
\cN(S_n;C(\sigma),C(\sigma),C(\nu))
=
\cfrac{(n-1)!\cdot(n-1)!\cdot|C(\nu)|}{n!}
\,\cdot\\ \cdot
\sum_{j=0}^{n-1}
(-1)^j\cdot(-1)^j\cdot\chi_j(C(\nu))\cdot\frac{j!(n-1-j)!}{(n-1)!}
=\\=
\cfrac{|C(\nu)|}{n}
\,\cdot
\sum_{j=0}^{n-1} j!\,(n-1-j)! \cdot \chi_j(C(\nu))
\end{multline}
Plugging the expression~\eqref{eq:Frobenius:Formula:2}
into~\eqref{eq:N:m1:mr}
with $|C(\nu)|$ replaced by its value~\eqref{eq:card:C:m}
and applying~\eqref{eq:contribution:numbered}
we complete the proof of Theorem~\ref{th:contribution:all:1:cyl}.
\end{proof}

The latter formula becomes particularly simple in the case of the
minimal stratum  $\cH(2g-2)$  when  $C(\nu)=C(\sigma)$ and in the
case of the principal  stratum  $\cH(1,\dots,1)$  when the cyclic
decomposition of $\nu$ is composed of $2g-2$ cycles of length $2$.
\medskip

\begin{proof}[Proof of
Corollary~\ref{cor:total:contribution:min:and:principal}
for the minimal stratum $\cH(2g-2)$.]

In the case of the minimal stratum we get
\begin{equation}
\label{eq:Frobenius:Formula:long:cycle}
\cN(S_n;C(\sigma),C(\sigma),C(\sigma))
=
\cfrac{(n-1)!}{n}
\,\cdot
\sum_{j=0}^{n-1}(-1)^j j!\,(n-1-j)!\,.
\end{equation}
Using the combinatorial identity
$$
\sum_{k=0}^m \frac{(-1)^k}{\binom{x}{k}}=
\frac{x+1}{x+2}\left(1+\frac{(-1)^m}{\binom{x+1}{m+1}}\right)
$$
(see~(2.1) in~\cite{Gould}) we can simplify~\eqref{eq:Frobenius:Formula:long:cycle}
as
\begin{equation}
\label{eq:Frobenius:Formula:long:cycle:answer}
\cN(S_n;C(\sigma),C(\sigma),C(\sigma))
=\begin{cases}
2\cdot\cfrac{\big((n-1)!\big)^2}{n+1}&\text{for odd }n\\
0&\text{for even }n
\end{cases}
\end{equation}
Plugging the expression~\eqref{eq:Frobenius:Formula:long:cycle:answer}
into~\eqref{eq:N:m1:mr} and applying~\eqref{eq:contribution:numbered}
we complete the proof of formula~\eqref{eq:contribution:minimal}.
\end{proof}

The Lemma below will be used in the proof of
Corollary~\ref{cor:total:contribution:min:and:principal}.

\begin{Lemma}
The following identity is valid
\begin{equation}
\label{eq:combinatorial:identity}
\sum_{k=0}^m (-1)^k \left(\frac{\binom{m}{k}}{\binom{2m+1}{2k}}-\frac{\binom{m}{k}}{\binom{2m+1}{2k+1}}\right)
=
\begin{cases}
0\,,&\text{ when $m$ is even}\\
2\cdot\frac{m+1}{m+2}\,,&\text{ when $m$ is odd\,.}
\end{cases}
\end{equation}
\end{Lemma}
\begin{proof}
We use the  following combinatorial identities (see~(4.22) and~(4.23):
in~\cite{Gould})
\begin{align*}
S(m)&:=\ \sum_{k=0}^m (-1)^k \frac{\binom{m}{k}}{\binom{2m}{2k}}\ \quad =\
\frac{1+(-1)^m}{2}\cdot\frac{2m+1}{m+1}\\
T(m)&:=\ \sum_{k=0}^m (-1)^k \frac{\binom{m}{k}}{\binom{2m+1}{2k+1}}\ =\
\frac{1-(-1)^m}{2}\cdot\frac{1}{m+2} +\ (-1)^m\ =\\
&=\left\{\begin{array}{ll}
\frac{1}{m+2} - 1 & \text{if $m$ odd} \\
1                 & \text{if $m$ even.}
\end{array}\right.
\end{align*}
The second term in the sum~\eqref{eq:combinatorial:identity} is exactly $T(m)$,
while the first one
can be expressed in terms of $S(m)$ and $T(m)$ as follows:
\begin{multline*}
\sum_{k=0}^m (-1)^k \frac{\binom{m}{k}}{\binom{2m+1}{2k}}
=
\sum_{k=0}^m (-1)^k \frac{\binom{m}{k}}{\binom{2m}{2k}}\cdot\frac{2m+1-2k}{2m+1}
=\\=
\sum_{k=0}^m (-1)^k \frac{\binom{m}{k}}{\binom{2m}{2k}}\cdot
\left(\frac{2m+2}{2m+1}-\frac{2k+1}{2m+1}\right)
=\\=
\frac{2m+2}{2m+1}\cdot \sum_{k=0}^m (-1)^k \frac{\binom{m}{k}}{\binom{2m}{2k}}
\ -\
\sum_{k=0}^m (-1)^k \frac{\binom{m}{k}}{\binom{2m+1}{2k+1}}
=\\=
\frac{2m+2}{2m+1}\cdot S(m) - T(m)\,.
\end{multline*}
Plugging the values of $S(m)$ and of $T(m)$ into
the above expression we complete
the proof of the combinatorial identity~\eqref{eq:combinatorial:identity}.
\end{proof}

\begin{proof}[Proof of
Corollary~\ref{cor:total:contribution:min:and:principal}
for the principal stratum $\cH(1,\dots,1)$]
In the case of the principal stratum
we have $C(\nu)=C(\tau)$, where
$$
\tau=(1,2)(3,4)\dots (n-1,n)\qquad\text{and}\qquad
n=4g-4\,
$$
(see equation~\eqref{eq:n} for the formula for $n$). One has
$$
\chi_j(\tau)=(-1)^{[(j+1)/2]} \binom{n/2-1}{[j/2]}
$$
(see  the formula below~(A.26) in~\cite{Zagier}). Finally, it is easy
to    see   directly   that   $|C(\tau)|=(n-1)!!$.   Thus,   we   can
rewrite~\eqref{eq:Frobenius:Formula:2} in this particular case as
\begin{multline*}
\cN(S_n;C(\sigma),C(\sigma),C(\tau))
=\\=
\cfrac{(n-1)!!}{n}
\,\cdot
\sum_{j=0}^{n-1} j!\,(n-1-j)! \cdot
(-1)^{\left[\frac{j+1}{2}\right]}
\begin{pmatrix}\frac{n}{2}-1\\\left[\frac{j}{2}\right]\end{pmatrix}
=\\=
\cfrac{(n-1)!!}{n}
\,\cdot
(n-1)! \sum_{j=0}^{n-1}
(-1)^{\left[\frac{j+1}{2}\right]}\,\cdot\
\frac{\begin{pmatrix}\frac{n}{2}-1\\\left[\frac{j}{2}\right]\end{pmatrix}}{\binom{n-1}{j}}\,.
\end{multline*}
Denoting $m=\frac{n}{2}-1$, we rewrite the above sum as
$$
\sum_{j=0}^{n-1}
(-1)^{\left[\frac{j+1}{2}\right]}\,\cdot\
\frac{\begin{pmatrix}\frac{n}{2}-1\\\left[\frac{j}{2}\right]\end{pmatrix}}{\binom{n-1}{j}}
=
\sum_{k=0}^m (-1)^k \left(\frac{\binom{m}{k}}{\binom{2m+1}{2k}}-\frac{\binom{m}{k}}{\binom{2m+1}{2k+1}}\right)
$$
Recall that $n=4g-4$, so $m=2g-3$ is odd.
Applying formula~\eqref{eq:combinatorial:identity} we obtain
$$
\cN(S_n;C(\sigma),C(\sigma),C(\tau))
=
\cfrac{(n-1)!!}{n}\cdot(n-1)!
\cdot\left(2\cdot\frac{m+1}{m+2}\right)
\,.
$$
Thus,  the weighted number $\cN(1,\dots,1)$ of $1$-cylinder diagrams
(see~\eqref{eq:N:m1:mr})
for    the    principal   stratum
$\cH(1,\dots,1)$ in genus $g$, when $n=4g-4$ equals
\begin{multline*}
\cN(1,\dots,1)=
\frac{1}{n!}\cdot\cN(S_n;C(\sigma),C(\sigma),C(\tau))
=\\=
\frac{1}{(4g-4)!}\cdot
\frac{(4g-5)!!}{(4g-4)}\cdot(4g-5)!\left(2\cdot\frac{2g-2}{2g-1}\right)
=\\=
\frac{(4g-5)!!}{(4g-4)(2g-1)}=
\frac{(4g-5)!}{(2g-1)!}\cdot 2^{-(2g-2)}\,.
\end{multline*}
Applying~\eqref{eq:contribution:numbered}
we complete the proof of formula~\eqref{eq:contribution:principal}.
\end{proof}

We complete this section with the proof of
Proposition~\ref{pr:proportion:hyp}.

\begin{proof}[Proof of Proposition~\ref{pr:proportion:hyp}]
The  results  in~\cite{AEZ:genus:0}  provide the  exact values for the
hyperelliptic  connected  components  (and,  more  generally, for all
hyperelliptic loci), namely:

\begin{align}
\label{eq:vol:hyp}
&\Vol\cH^{hyp}(2g-2)&=\cfrac{2\pi^{2g}}{(2g+1)!}\cdot
\cfrac{(2g-3)!!}{(2g-2)!!} \sim
\cfrac{1}{\pi^2 g}\left(\frac{\pi e}{2g+1}\right)^{2g+1}\,.
\\
&\Vol\cH^{hyp}(g-1,g-1)&=\cfrac{4\pi^{2g}}{(2g+2)!}\cdot
\cfrac{(2g-2)!!}{(2g-1)!!} \sim
\cfrac{1}{\pi^2 g}\left(\frac{\pi e}{2g+2}\right)^{2g+2}\,.
\end{align}

There is a single $1$-cylinder separatrix diagram for any
hyperelliptic connected component $\cH^{hyp}(2g-2)$ or
$\cH^{hyp}(g-1,g-1)$. Proposition~\ref{pr:contribution:Abelian}
provides the contribution of this diagram to the volume. Taking the
ratio of the resulting expressions~\eqref{eq:contribution:numbered}
and~\eqref{eq:vol:hyp} we obtain the expressions claimed in
Proposition~\ref{pr:proportion:hyp}.
\end{proof}

\section{Alternative counting of $1$-cylinder separatrix diagrams}
\label{s:Alternative:counting}

In this section we suggest two alternative methods of
counting $1$-cylinder separatrix diagrams. The first one, elaborated in
section~\ref{ss:recursive:relations}, is based
on recursive relations  for the numbers of such
diagrams. The second method, presented in
section~\ref{ss:Rauzy:classes}, uses Rauzy diagrams and admits simple
computer realization for low-dimensional strata.

\subsection{Approach based on recursive relations}
\label{ss:recursive:relations}

Here we explicitly enumerate $1$-cylinder separatrix diagrams that give
rise  to  Abelian  differentials  (orientable  case)  or to quadratic
differentials  (nonorientable case) with 0, 1 or 2 saddle connections
shared between the two boundary components of the cylinder.
\medskip

\noindent\textbf{Strata of Abelian differentials.}
We start with the case of orientable separatrix diagrams;
they represent strata of Abelian differentials.
Take a cylinder whose boundary
components are two identical copies of an $\nofint$-gon
\textit{with a marked side}.
Choose an orientation of the cylinder and consider the induced orientation
on its boundary components.
Consider a gluing that identifies the sides of one boundary polygon
with the sides of the other reversing their orientation and
respecting the marked sides. We get a closed orientable surface with
a connected graph $\Gamma$ (the image of the cylinder boundary
components) embedded into it. All vertices of $\Gamma$ have even degree,
and we denote by $v_i$ the number of vertices of $\Gamma$ of degree $2i$.
Clearly, $\nofint=\sum_{i\geq 1} iv_i$, and we call
$[1^{v_1}2^{v_2}\dots]$ the \textit{type} of the cylinder gluing.
The associated $1$-cylinder separatrix diagram corresponds
to the stratum $\cH(0^{v_1},1^{v_2},2^{v_3},\dots)$, and
the complex dimension of this stratum is $n+1$. We warn the
reader that the degrees of zeros and the indexation of
their multiplicities is shifted by one: there are $v_{j+1}$
zeroes of degree $j$. Such indexation of the entries of the
partition $\nu$ is more natural for combinatorial
operations with the associated graphs extensively performed
in this section.

Let us now fix a partition $\nu=[1^{v_1}2^{v_2}\dots]$ of $\nofint$ and
denote by $N_\nofint(\nu)$ the number of cylinder gluings of type $\nu$
described above. Consider the generating functions
\begin{align}
&F_\nofint(t_1,t_2,\dots)=\sum_{\nu\,\vdash\, \nofint} N_{\nofint}(\nu)\,t_1^{v_1}\,t_2^{v_2}\dots\;,\nonumber\\
&F(s;t_1,t_2,\dots)=\sum_{\nofint\geq 1}s^{\nofint-1}\,F_\nofint(t_1,t_2,\dots)\;.\nonumber
\end{align}

\begin{thm}\label{abel}
Put
\begin{align}
M_1=\sum_{i=2}^\infty \sum_{j=1}^{i-1} (i-1)t_j t_{i-j}\,\frac{\partial}{\partial t_{i-1}} + j(i-j) t_{i+1}\,\frac{\partial^2}{\partial t_j \partial t_{i-j}}\;.
\end{align}
Then the generating function $F=F(s;t_1,t_2,\dots)$ satisfies the linear PDE
\begin{align}\label{pde}
\frac{\partial F}{\partial s}=M_1F
\end{align}
and is uniquely determined by the initial condition $F|_{s=0}=t_1$.
Equivalently, the generating function $F$ is explicitly given by the formula
\begin{align}
F(s;t_1,t_2,\dots)=e^{sM_1}t_1\;.
\end{align}
\end{thm}

\begin{proof}
First, rewrite~\eqref{pde} as a recursion for the numbers
$N_{\nofint}(\nu)$. Denote by $\e_i$ the sequence with 1 at the $i$-th
place and 0 elsewhere. Then~\eqref{pde} is equivalent to
\begin{align}
(\nofint-1)&N_{\nofint}(\nu)=\nonumber\\
&=\sum_{i=2}^\infty \sum_{j=1}^{i-1} (i-1)(v_{i-1}+1-\delta_{j,1}-\delta_{i-j,1})\,N_{\nofint}(\nu-\e_j-\e_{i-j}+\e_{i-1})+\nonumber\\
&+\sum_{i=2}^\infty \sum_{j=1}^{i-1} j(i-j)(v_j+1)(v_{i-j}+1+\delta_{j,i-j})\,N_{\nofint}(\nu+\e_j+\e_{i-j}-\e_{i-1})\label{rec}\;.
\end{align}
We prove it by establishing a direct bijection between cylinder
gluings counted in the left and right hand sides of (\ref{rec}).
Consider the ribbon graph $\Gamma^*$ dual to $\Gamma$. It has 2 vertices
(each of degree $\nofint$) and $\nofint$ edges connecting these two vertices (one
of these edges is marked). Let us pick a non-marked edge in $\Gamma^*$,
this can be done in $(\nofint-1)$ ways giving the l.h.s. in (\ref{rec}).
Deletion of this edge results in one of the following two
possibilities:
\begin{enumerate}[label=\roman*)] 
\item The edge belongs to two different boundary cycles of $\Gamma^*$
    of lengths $2j$ and $2(i-j)$. The edge deletion gives rise to
    one boundary cycle of length $2(i-1)$ and the graph type
    changes to $\nu-\e_j-\e_{i-j}+\e_{i-1}$.
\item One boundary cycle of length $2(i+1)$ traverses the edge
    twice (once in each direction). After the edge deletion the
    boundary cycle splits into two ones of lengths $2i$ and
    $2(i-j)$ and the graph type changes to
    $\nu+\e_j+\e_{i-j}-\e_{i+1}$.
\end{enumerate}
Counting the number of ways that each case can occur we get the first
and the second sums in~\eqref{rec} respectively.

To show that the generating function $F$ is uniquely determined by
the initial condition $F|_{s=0}=t_1$, we first notice that $F_1=t_1$
(for $\nofint=1$ there is only one $1$-cylinder configuration). The equation
(\ref{pde}) recursively expresses $F_\nofint$ in terms of $F_{\nofint-1}$ as
follows:
\begin{align}
(\nofint-1)\,F_\nofint=M_1F_{\nofint-1}\;.\label{ind}
\end{align}
The formula $F=e^{sM_1}t_1$ is just another way of writing the same thing.
\end{proof}

\begin{rem}
The numbers $N_\nofint(\nu)$ giving the {\em rooted} count of $1$-cylinder
configurations and the numbers
$\cN(0^{v_1},1^{v_2},2^{v_3},\dots)$,
see~\eqref{eq:N:m1:mr}, giving the {\em
weighted} count of $1$-cylinder diagrams
in $\cH(0^{v_1},1^{v_2},2^{v_3},\dots)$ with weights $1/|\Aut(\Gamma)|$ are related by the
simple formula
\begin{equation}
\label{eq:nonrooted:from:rooted}
\cN(0^{v_1},1^{v_2},2^{v_3},\dots)=
\frac{1}{\nofint}\cdot N_\nofint(\nu)\,.
\end{equation}
\end{rem}

Recall that by definition of the polynomial $F_n$ the
coefficient of the monomial $t_1^{v_1}\,t_2^{v_2}\,\cdots$
equals $N_\nofint(\nu)$, where $\nofint=\sum_{i\geq 1} iv_i=
\dim_\C\cH(0^{v_1},1^{v_2},2^{v_3},\dots)-1$.
\begin{Corollary}
The absolute contribution $c_1(\cLH)$ of
all $1$-cylinder square-tiled surfaces to the volume $\Vol\cLH$
of the stratum
$\cLH=\cH(1^{v_2},2^{v_3},\dots)$ of Abelian
differentials equals
\begin{equation}
\label{eq:c1:from:generating:function}
c_1(\cLH)=
\frac{2}{n!}\cdot
v_2!\cdot v_3!\cdots
\,\cdot \zeta(n+1)\cdot N_n(\nu)\,.
\end{equation}
\end{Corollary}
\begin{proof}
The contribution of a single 1-cylinder diagram $\Gamma$
is given by formula~\eqref{eq:contribution:numbered},
which in notations of the Corollary gives
$$
c(\Gamma)=\cfrac{2}{|\Aut(\Gamma)|}\cdot
\cfrac{v_2!\cdot v_3! \cdots}{(n-1)!}\,\cdot \zeta(n+1)\,.
$$
Combining this result with the weighted count~\eqref{eq:nonrooted:from:rooted}
of 1-cylinder diagrams we obtain~\eqref{eq:c1:from:generating:function}.
\end{proof}

\begin{Example}

Consider the generating functions for small values of $\nofint$:
\begin{align*}
F_1&=t_1\\
F_2&=t_1^2\\
F_3&=t_1^3+\boldsymbol{t_3}\\
F_4&=t_1^4+4 t_1 t_3 + t_2^2\\
F_5&=t_1^5+10t_3 t_1^2 + 5t_1 t_2^2 + 8t_5\\
F_6&=t_1^6+20 t_1^3 t_3 + 15 t_1^2 t_2^2 + 48 t_1 t_5 + \boldsymbol{24 t_2 t_4} + 12 t_3^2
\end{align*}
We know that there is a single $1$-cylinder diagram in the stratum
$\cH(2)$ which has symmetry of order $3$, see Figure~\ref{fig:diag}
in section~\ref{ss:separatrix:diagrams}. For this stratum we have
$\nu=[3^1]$ so we can read the weighted number of $1$-cylinder
diagrams from the coefficient in front of $t_3$ in $F_3$
normalizing it as in~\eqref{eq:nonrooted:from:rooted}. This gives
$c_1(\cH(2))=\frac{1}{3}\zeta(4)$ as expected, see~\eqref{eq:contributions:1:2:for:H2}.

Consider now the stratum $\cH(3,1)=\cH(1^1,3^1)$.
It has dimension $\dim_{\C{}}\cH(3,1)=7$,
so $n=6$.
The number of associated rooted diagrams is given
by  the  coefficient  of  the monomial $24 t_2 t_4$ in the polynomial
$F_6$. Applying~\eqref{eq:c1:from:generating:function} we get the following impact
of all $1$-cylinder square-tiled surfaces to the volume of this stratum:
$$
c_1(\cH(3,1))=
\frac{2}{6!}\cdot 1!\cdot 1!\cdot\zeta(7)\cdot 24 =
\frac{1}{15}\cdot \zeta(7)\,.
$$
By~\cite{Eskin:Masur:Zorich} we have
$$
\Vol\cH(3,1)=\frac{16}{42525}\pi^6=\frac{16}{45}\zeta(6)\,.
$$
Thus, the relative impact $\prop_1(\cH(3,1))$
of $1$-cylinder diagrams is equal to
$$
\left(\frac{1}{15}\zeta(7)\right):\left(\frac{16}{45}\zeta(6)\right)=
\frac{3\zeta(7)}{16\zeta(6)}\,.
$$
which matches the value given in Example~\ref{ex:H31}.

\end{Example}
\medskip

\noindent\textbf{Strata of quadratic differentials.}
Now we proceed with with the case of nonorientable separatrix diagrams;
they represent strata of meromorphic quadratic
differentials with at most simple poles. Take a cylinder bounded by
two polygons, one with $l+2m$ sides and the other with $l+2n$ sides
and consider its orientable gluings that identify $m$ pairs of sides
of the first polygon, $n$ pairs of sides of the second polygon, and
$l$ sides of the first one with $l$ sides of the second one.

We warn the reader that we have two polygons
with a priori different number of sides, and that from now on
the symbol $n$ does not denote the total number of sides anymore.
Contrary to the previous section we
do not mark any side on either of the two polygons anymore.

We get a closed orientable surface, and the image of the
boundary polygons is a graph $\Gamma$ (not necessarily
connected) embedded into it. Suppose that $\Gamma$ has the
vertex degree set $v_1,v_2,\ldots$, where
$\nu=[1^{v_1}\,2^{v_2}\,\ldots]$ is a partition of
$2(l+m+n)$ (this means that $\Gamma$ has $v_1$ vertices of
degree 1, $v_2$ vertices of degree 2, etc.). The associated
$1$-cylinder separatrix diagram corresponds to the stratum
$\cQ(-1^{v_1},0^{v_2},1^{v_3},\dots)$, and the complex
dimension of this stratum is $l+m+n$. Note that this time
the degrees of zeros and the indexation of their
multiplicities is shifted by two: meromorphic quadratic
differentials under consideration have $v_{j+2}$ zeroes of
degree $j$, where ``zero of degree $-1$'' is a simple pole,
and ``zero of degree $0$'' is a marked point.

Denote by $N_{l,m,n}(v_1,v_2,\ldots)$ the weighted
count of such gluings. It coincides with the number
$\cN_{l,m,n}(-1^{v_1},0^{v_2},1^{v_3},\dots)$
giving the \textit{weighted} count of $1$-cylinder diagrams
of type $(l,m,m)$ in $\cQ(-1^{v_1},0^{v_2},1^{v_3},\dots)$
with weights $1/|\Aut(\Gamma)|$
up to a correction in the symmetric
case when $m=n$:
\begin{equation}
\label{eq:number:of:l:m:n:diagrams}
\cN_{l,m,n}(-1^{v_1},0^{v_2},1^{v_3},\dots)
=
\begin{cases}
N_{l,m,n}(v_1,v_2,\ldots)&\text{when $m\neq n$}\\
\frac{1}{2}\cdot N_{l,m,n}(v_1,v_2,\ldots)&\text{when $m= n$\,.}
\end{cases}
\end{equation}

Consider the generating series
\begin{align}
F_{l,m,n}=\sum_{\nu\vdash 2(l+m+n)}N_{l,m,n}(v_1,v_2,\ldots)p_1^{v_1}p_2^{v_2}\ldots\;.
\end{align}

To explicitly compute $F_{l,m,n}$ with $l=0,1,2$ we
introduce an auxiliary generating series
$G(s,p_1,p_2,\ldots)$. The coefficient of $G$ at the
monomial $s^{2\nofsides}p_1^{v_1}p_2^{v_2}\ldots$ is the
number of orientable gluings of a
$2\nofsides$-gon with fixed vertex degree set
given by the partition $[1^{v_1}\,2^{v_2}\,\ldots]$ of
$2\nofsides$. In other words, each gluing
produces a closed orientable surface of genus
$g=\frac{1}{2}\left(1+\nofsides-\sum_i
v_i\right)$ together with a graph embedded into it with
$v_1$ vertices of degree 1, $v_2$ vertices of degree 2,
etc. As usual, the gluings are counted with weights
reciprocal to the orders of the automorphism groups.

The generating series $G(s,p_1,p_2,\ldots)$ was extensively studied
in ~\cite{Kazarian:Zograf}. In particular, as it follows from Theorem
3 (ii) in~\cite{Kazarian:Zograf}, the series $G$ is uniquely
determined by the equation
\begin{align}
\frac{1}{s}\frac{\partial G}{\partial s}=M_2G+p_1^2
\end{align}
modulo the initial condition $G|_{s=0}=0$,
where
\begin{align}
M_2=\sum_{i=2}^\infty \sum_{j=1}^{i-1}(i-2)p_j p_{i-j}\,\frac{\partial}{\partial p_{i-2}} + j(i-j) p_{i+2}\,\frac{\partial^2}{\partial p_j \partial p_{i-j}}\;.
\end{align}

It will be convenient to write $G$ as a power series in $s$:
\begin{align}
G(s,p_1,p_2,\ldots)=\sum_{\nofsides=1}^\infty s^{2\nofsides} G_{\nofsides}(p_1,p_2,\ldots)\;.
\end{align}
Then we have

\begin{thm}\label{thm:quad}
The following formulas hold:
\begin{align}
F_{0,m,n}&=G_m G_n\;,\label{0}\\
F_{1,m,n}&=\sum_{i=1}^\infty \sum_{j=1}^{\infty}ij\, p_{i+j+2}\,\frac{\partial G_m}{\partial p_i}\frac{\partial G_n}{\partial p_{j}}\;,\label{1}\\
F_{2,m,n}&=\frac{1}{2}\sum_{i=1}^\infty\sum_{j=1}^\infty\sum_{k=1}^\infty\sum_{\ell=1}^\infty ijk\ell\,p_{i+k+2}\,p_{j+\ell+2}
\frac{\partial^2 G_m}{\partial p_i \partial p_j}\frac{\partial^2 G_n}{\partial p_k \partial p_{\ell}}\label{2}\\
&+\sum_{i=1}^\infty\sum_{j=1}^\infty\sum_{k=1}^\infty ijk(k+1)\,p_{i+j+k+4}\left(\frac{\partial^2 G_m}{\partial p_i \partial p_j}\frac{\partial G_n}{\partial p_k}+\frac{\partial G_m}{\partial p_k}\frac{\partial^2 G_n}{\partial p_j \partial p_k}\right)\nonumber\\
&+\sum_{i=1}^\infty\sum_{j=1}^\infty ij \left(\sum_{k=0}^i\sum_{\ell=0}^j p_{k+\ell+2}\,p_{i+j+2-k-\ell}\right)\frac{\partial G_m}{\partial p_i}\frac{\partial G_n}{\partial p_j}\;.\nonumber
\end{align}
\end{thm}

\begin{proof}
Instead of the graph $\Gamma$ (the image of cylinder's boundary) it is
handier to consider its dual graph $\Gamma^*$. The graph $\Gamma^*$ has two
vertices, $m$ loops incident to the first vertex, $n$ loops incident
to the second vertex and $l$ edges connecting the first vertex with
the second one. We also assume that the vertices are labeled.

Formula (\ref{0}) of Theorem~\ref{thm:quad} is obvious.

To prove (\ref{1}), let us take two ribbon graphs with one vertex
each, the first one with $m$ loops and the second one with $n$ loops.
Let us count the number of ways to connect the two vertices with a
single edge. For any boundary component of length $i$ of the first
graph and any boundary component of length $j$ of the second graph
there are $ij$ possibilities to connect them with an edge. Instead of
two disjoint boundary components of lengths $i$ and $j$ we get a
single boundary component of length $i+j+2$. This simple observation
is precisely described by Formula (\ref{1}).

The proof of Formula (\ref{2}) is similar to that of (\ref{1}).
Again, we start with two ribbon graphs with one vertex each, the
first one with $m$ loops and the second one with $n$ loops. Now we
count the number of different ways to connect the two vertices with a
double edge. Four possibilities can occur:
\begin{enumerate}[label=\roman*)]
\item Two different boundary components of the first graph of lengths
$i$ and $j$ are connected by two edges with two boundary components
of the second graph of lengths $k$ and $\ell$ respectively. There are
$ijk\ell$ ways to do that. The boundary components of lengths $i$ and
$k$ are replaced by a single boundary component of length $i+k+2$,
and the components of lengths $j$ and $\ell$ are replaced by a single
component of length $j+\ell+2$. This possibility is described by the
first line in the right hand side of (\ref{2}).
\item Two different boundary components of the first graph of
lengths $i$ and $j$ are connected by two edges with
one boundary components of the second graph of lengths $k$. This can
be done in $ijk(k+1)$ ways. The three boundary components of
lengths $i,\;j$ and $k$ are replaced by a single
boundary component of length $i+j+k+4$.
\item A boundary component of the first graph of length $k$ is
connected by two edges with two boundary components of the second
graph of lengths $i$ and $j$. Similar to the previous case, this can
be done in $ijk(k+1)$ ways. The three boundary components of
lengths $i,\;j$ and $k$ are replaced by a single
boundary component of length $i+j+k+4$. The cases (ii) and (iii) can
be united to produce the second line in the right hand side of
(\ref{2}).
\item A boundary component of the first graph of length $i$ is
connected by two edges with a boundary component of the second graph
of length $j$. There are $ij$ ways to connect the two boundary
components with one edge. If the endpoints of the second edge at the
distances $k$ and $\ell$ from the endpoints of the first one, the
components of lengths $i$ and $j$ get replaced by the boundary
components of lengths $k+\ell+2$ and
$i+j+2-k-\ell$. This last possibility is described by the third line
in the right hand side of (\ref{2}).
\end{enumerate}
\end{proof}

\begin{Example}
\label{ex:Qi313}
To find the contribution of $1$-cylinder separatrix diagrams to the
volume of the stratum $\cQ(1^3,-1^3)$ we have to find the weighted
number of ribbon graphs as above with $3$ vertices of valence $1$
(corresponding to $3$ simple poles) and with $3$ vertices of valence
$3$ (corresponding to $3$ simple zeroes). So the \textit{type of the
cylinder gluing} representing the stratum $\cQ(1^3,-1^3)$ is $[1^3,
3^3]$ and we are interested in monomials corresponding to $p_1^3
p_3^3$ in polynomials $F_{l,m,n}$ with $l+m+n=6$. We present some of
them to compare the result with the diagram-by-diagram calculation
presented in the next section.
\begin{align}
\label{eq:F015}
F_{0,1,5}&=4p_1^3 p_4 p_2 p_3 + p_1^5 p_2^2 p_3 + 3p_1^3 p_5 p_2^2 + \frac{1}{2} p_1^6 p_2 p_4 + 5p_1^4 p_6 p_2 + \frac{7}{2} p_1^4 p_5 p_3 + \frac{1}{10} p_1^7 p_5\\
\notag
&+\frac{5}{2} p_1^5 p_7 + \frac{21}{2} p_9 p_1^3 + \frac{21}{4} p_8 p_1^2 p_2 + \frac{7}{2} p_1^2 p_7 p_3 + \frac{13}{4} p_1^2 p_4 p_6 + \frac{33}{20} p_1^2 p_5^2 + \frac{1}{4} p_1^4 p_2^4\\
\notag
&+\frac{1}{4} p_1^6 p_3^2 + \boldsymbol{\frac{1}{2} p_1^3 p_3^3} + \frac{1}{2} p_1^2 p_4 p_2^3 + \frac{1}{2} p_1^2 p_2^2 p_3^2 + \frac{3}{2} p_1^4 p_4^2\,.\\
\label{eq:F033}
F_{0,3,3}&=\frac{1}{3}p_4 p_1^3 p_2 p_3 + p_4 p_1 p_2 p_5 + \frac{1}{36} p_3^4 + \frac{1}{3} p_1^5 p_2^2 p_3 + \frac{1}{6} p_1^2 p_2^2 p_3^2\\
\notag
&+\frac{1}{6} p_3^2 p_4 p_2 + \frac{1}{3}  p_1 p_5 p_3^2 + \frac{1}{4} p_1^4 p_2^4 + \frac{1}{9} p_1^6 p_3^2 + \boldsymbol{\frac{1}{9} p_1^3 p_3^3}\\
\notag
&+\frac{1}{4} p_2^2 p_4^2 + p_5^2 p_1^2 + p_5 p_1^3 p_2^2 + \frac{1}{2} p_4 p_1^2 p_2^3 + \frac{2}{3} p_5 p_1^4 p_3\,.\\
\label{eq:F213}
F_{2,1,3}&=10p_4p_5p_2p_1 + 16 p_8 p_1^2 p_2 + 4p_7 p_2^2 p_1 + 13p_4 p_6 p_1^2 + 7 p_5^2 p_1^2 + 12 p_9 p_1^3 + 5 p_{10} p_2\\
\notag
&+36 p_{11} p_1 + \frac{1}{2} p_4^2 p_2^2 + 5 p_1^3 p_2 p_3 p_4 + 5 p_1 p_2 p_3 p_6 + \boldsymbol{p_3^3 p_1^3} + 3 p_1 p_5 p_3^2 + 4 p_1 p_3 p_4^2\\
\notag
&+2 p_3 p_9 + \frac{13}{2} p_4 p_8 + 5 p_5 p_7 + \frac{3}{2} p_6^2 + p_1^4 p_4^2 + p_1^2 p_2^3 p_4 + p_1^2 p_2^2 p_3^2 + 2 p_1^3 p_2^2 p_5\\
\notag
& + p_1^4 p_2 p_6 + 2 p_1^4 p_3 p_5 + 13 p_1^2 p_3 p_7\,.
\end{align}
By~\eqref{eq:F015} the term $p_1^3 p_3^3$ in $F_{0,1,5}$ has
coefficient $\frac{1}{2}$, so the weighted number $\sum_\cD
\frac{1}{\Aut(\cD)}$ of $1$-cylinder diagrams representing the
stratum $\cQ(1^3,-1^3)$ with $l=0,m=1,n=5$ is equal to $\frac{1}{2}$.
Table~\ref{tab:Qi3:13}  in section~\ref{ss:Rauzy:classes} shows that such diagram
is, actually, unique, and that its symmetry group $\Aut(\cD)$
indeed has order $2$.

By~\eqref{eq:F033} the term $p_1^3 p_3^3$ in $F_{0,3,3}$ has
coefficient $\frac{1}{9}$, so the weighted number $\sum_\cD
\frac{1}{\Aut(\cD)}$ of $1$-cylinder diagrams representing the
stratum $\cQ(1^3,-1^3)$ with $l=0,m=3,n=3$ is equal to $\frac{1}{18}$
(recall that when $m=n$ we have to divide the corresponding coefficient
by $2$ to get the weighted number of diagrams; see~\eqref{eq:number:of:l:m:n:diagrams}).
Table~\ref{tab:Qi3:13} in section~\ref{ss:Rauzy:classes} shows that there is a
unique such diagram, and that its symmetry group $\Aut(\cD)$ has
order $18$.

By~\eqref{eq:F213} the term $p_1^3 p_3^3$ in $F_{2,1,3}$ has
coefficient $1$.
Table~\ref{tab:Qi3:13} in section~\ref{ss:Rauzy:classes} shows that there is a
unique $1$-cylinder diagram with $l=2,m=1,n=3$ in the stratum $\cQ(1^3,-1^3)$,
and that this diagram does not have any symmetries.

\end{Example}

\subsection{Approach based on Rauzy diagrams}
\label{ss:Rauzy:classes}

As can be seen from Theorem~\ref{thm:quad}, the generating functions
of one-cylinder diagrams in quadratic strata of Abelian differentials
are complicated. In this section we consider an alternative approach
to list all $1$-cylinder separatrix diagrams in a given
component stratum of meromorphic quadratic differentials
$\cQ(d_1,\dots,d_k)$ with at most simple poles. The method is mostly
suited for computational purposes when the stratum has relatively
small dimension.

As before, we denote by $\mult_{-1},\mult_1,\mult_2,\dots$ the multiplicities
$\mult_{j}$ of entries $j\in\{-1,1,2,\dots\}$ in the set
$\{d_1,\dots,d_k\}$, where $\sum d_i =4g-4$, and $g\in\Z_+$.
In the notation of section~\ref{ss:recursive:relations} we have
$\mult_{i}=v_{i+2}$.

\textit{Rauzy diagrams} are strongly connected oriented graphs whose vertices are
generalized permutations already considered in
Section~\ref{ss:Application:experimental:evaluation:of:MV:volumes}.
There is a bijection between Rauzy diagrams of generalized
permutations and connected components of strata,
see~\cite{Boissy:Lanneau} and~\cite{Veech:Gauss:measures}. Moreover,
any $1$-cylinder diagram in the corresponding component is
represented by a certain subcollection of generalized permutations
whose top first and bottom last symbols are identical; such
(generalized) permutations are called \textit{standard} permutations
in the context of Rauzy diagrams.

Figure~\ref{fig:Jenkins:Strebel} at the beginning of
section~\ref{ss:contribution:of:one:1:cylinder:diagram:computation} illustrates
how the standard generalized permutation
$$
\begin{pmatrix}&0&1&1&\\ &2&3&2&3&0\end{pmatrix}
$$
represents a nonorientable $1$-cylinder separatrix diagram.
The bottom picture in Figure~\ref{fig:merging:zeroes} from
section~\ref{ss:1:cylinder:diagrams:Abelian} illustrates how the
standard permutation
$$
\begin{pmatrix}
0& 1& 2& 3& 4& 5& 6& 7& 8\\
4& 3& 2& 5& 8& 7& 6& 1& 0
\end{pmatrix}
$$
represents the orientable $1$-cylinder diagram on top of
Figure~\ref{fig:merging:zeroes}.

It is very easy to generate all permutations in a Rauzy diagram associated
to any low-dimensional stratum. Given a stratum of meromorphic quadratic differentials with
at most simple poles, say, $\cQ(1^3, -1^3)$, we first use the
method~\cite{Zorich:representatives} of one of the authors to
construct some generalized permutation representing the desired
(connected component of) the stratum. Next, one just has to apply two
simple transformation rules to generate the whole Rauzy diagram from
any element. Using the \texttt{surface\_dynamics} package of the
software SageMath it is a five line program to get the list of the
158 standard permutations in $\cQ(1^3, -1^3)$:
\begin{verbatim}
sage: from surface_dynamics.all import *
sage: Q = QuadraticStratum({1:3, -1:3})
sage: p = Q.permutation_representative()
sage: R = p.rauzy_diagram(right_induction=True, left_induction=True)
sage: R
Rauzy diagram with 2010 permutations
sage: std_perms = [q for q in R if q[0][0] == q[1][-1]]
sage: len(std_perms)
158
\end{verbatim}

Note that the same $1$-cylinder separatrix diagram might be (and
usually is) represented by several standard generalized permutations.
For example, the following four standard generalized
permutations represent the same $1$-cylinder separatrix diagram:
\begin{equation}
\label{eq:all:standard:permutations}
\begin{pmatrix}
0\,1\,2\,3\,1\,2\,3\\
4\,4\,5\,5\,6\,6\,0
\end{pmatrix}
\quad
\begin{pmatrix}
0\,1\,2\,3\,1\,2\,3\\
4\,5\,5\,6\,6\,4\,0
\end{pmatrix}
\quad
\begin{pmatrix}
0\,1\,1\,2\,2\,3\,3\\
4\,5\,6\,4\,5\,6\,0
\end{pmatrix}
\quad
\begin{pmatrix}
0\,1\,2\,2\,3\,3\,1\\
4\,5\,6\,4\,5\,6\,0
\end{pmatrix}.
\end{equation}
We can put standard permutations into the one-to-one correspondence
with $1$-cylinder separatrix diagrams endowing the latter with the
following extra structure. Choose one of the two possible choices of
a top and a bottom boundary component of the cylinder, and mark
a saddle connection on each boundary component.

This multiplicity is directly related to the cardinality of the
automorphism group $|\Aut(\cD)|$ that we discuss now.
All standard generalized permutations representing any given
separatrix diagram $\cD$, can be obtained from any standard
generalized permutations representing $\cD$ by the following two
operations.

Remove distinguished symbols (denoted by ``$0$'' in the examples
above); rotate cyclically the top line by any rotation; rotate
cyclically the bottom line by any rotation; insert the distinguished
element on the left of the upper line and on the right of the bottom
one; renumber the entries. We get a collection $D_1$ of standard
generalized permutations.

For example, the second generalized
permutation in~\eqref{eq:all:standard:permutations} is
obtained from the first one by cyclically shifting by one
position to the left the elements $4\,4\,5\,5\,6\,6$ of the
bottom line keeping the symbol $0$ fixed. Applying the same
operation one more time and renumbering the elements we
return to the first permutation
in~\eqref{eq:all:standard:permutations}. Finally, the
analogous operation applied to the top line of the first
permutation does not change the permutation (up to
renumbering the entries). Hence, in this example, $D_1$ is
composed of the first two permutations
in~\eqref{eq:all:standard:permutations}.

Apply to every standard generalized permutations in $D_1$ the
following operation. Remove distinguished symbols (denoted by ``$0$''
in the examples above); interchange the top and the bottom line;
insert the distinguished element on the left of the upper line and on
the right of the bottom one and renumber the entries. We get one more
collection $D_2$ of standard generalized permutations.
In the example~\eqref{eq:all:standard:permutations}
the set $D_2$ is composed of the third and forth permutations.

Take the union of $D_1$ and $D_2$. It is easy to see that we have
constructed all standard generalized permutations representing the
initial separatrix diagram $\cD$. We suggest to the reader to check
that the collection~\eqref{eq:all:standard:permutations} can be
constructed by the two operations as above from any of its elements.

Since the top boundary component is composed from $l+2m$ separatrices
and the bottom component from $l+2n$ ones, the cardinality of the
set of nontrivial operations as above is $2 \times (l+2m) \times
(l+2n)$. The factor 2 here stands for the inversion of the top and
bottom lines of the generalized permutation. Thus, the
order $|\Aut(\cD)|$ of the symmetry group
$\Aut(\cD)$ of the associated separatrix diagram $\cD$ is
$$
|\Aut(\cD)|:=\Big(2 \times (l+2m) \times (l+2n)\Big)/
\card(D_1\cup D_2)\,.
$$
In example~\eqref{eq:all:standard:permutations} we get
$$
|\Aut(\cD)|=\Big(2 \times (0+2\cdot 3) \times (0+2\cdot 3)\Big)/4=18
$$
as indicated in the second line in Table~\ref{tab:Qi3:13} where
$l = 0$, $m = 3$, $n = 3$.

\subsection{The example of $\cQ(1^3, -1^3)$}
To give an idea of an approximate calculation of the volume
based on our method we compute $\Vol\cQ(1^3,-1^3)$ (the stratum is
chosen by random). We present a list of all ribbon graphs $\cD$
satisfying the above conditions, which are realizable in
$\cQ(1^3,-1^3)$. For each such ribbon graph we give the
order $|\Aut|=|\Aut(\cD)|$ of its symmetry group, we
present $l,m,n$ and we apply
formula~\eqref{eq:general:contribution} to compute its contribution
to the volume of the stratum. Recall the convention used
in~\eqref{eq:general:contribution}: defining the symmetry group
$\Aut(\cD)$ we assume that none of the vertices, edges, or
boundary components of the ribbon graph $\cD$ is labeled; however,
we assume that the orientation of the ribbons is fixed.

The  stratum  $\cQ(1^3,-1^3)$  corresponds  to  genus  $g=1$.  It  is
connected      and      $d=\dim_\C\cQ(1^3,-1^3)=6$.      We      have
$\mult_{-1}=3$,   $\mult_1=3$,   and   there  are  no  other  entries
$\mult_k$.  This  means that every such ribbon graph has $3$ vertices
of valence one, and $3$ vertices of valence $3$.

\begin{table}[htb]
$$
\begin{array}{|c|c|c|c|}
\hline &&& \\ [-\halfbls]
\text{Ribbon graph }\cD\hspace*{5pt} &|\Aut(\cD)|& l,m,n &\text{Contribution to }\Vol\cQ(1^3,-1^3)\\
&&& \\ [-\halfbls]
\hline&&&\\
[-\halfbls]
&& l=0 &\\
\includegraphics{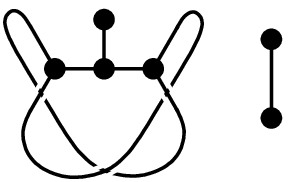}
& 2 & m=5 &
\cfrac{2^{0+2}}{2}\cdot\cfrac{(5+1-2)!}{(5-1)!(1-1)!}\cdot\cfrac{3!\cdot 3!}{(6-2)!}\,\zeta(6)=3\zeta(6)\\
&& n=1 &\\
\hline&&&\\
[-\halfbls]
&& l=0 &\\
\includegraphics{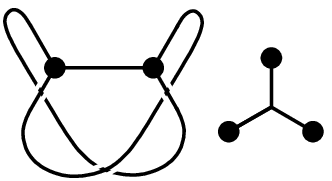}
& 18 & m=3 &
\cfrac{2^{0+2}}{18}\cdot\cfrac{(3+3-2)!}{(3-1)!(3-1)!}\cdot\cfrac{3!\cdot 3!}{(6-2)!}\,\zeta(6)=2\zeta(6)\\
&& n=3 &\\
\hline&&&\\
[-\halfbls]
&& l=2 &\\
\includegraphics{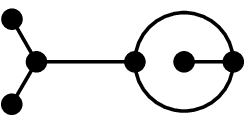}
& 1 & m=3 &
\cfrac{2^{2+2}}{1}\cdot\cfrac{(3+1-2)!}{(3-1)!\cdot(1-1)!}\cdot\cfrac{3!\cdot 3!}{(6-2)!}\,\zeta(6)=24\zeta(6)\\
&& n=1 &\\
\hline&&&\\
[-\halfbls]
&& l=3 &\\
\includegraphics{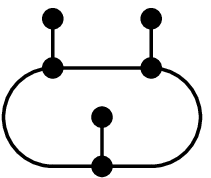}
& 1 & m=2 &
\cfrac{2^{3+2}}{1}\cdot\cfrac{(2+1-2)!}{(2-1)!(1-1)!}\cdot\cfrac{3!\cdot 3!}{(6-2)!}\,\zeta(6)=48\zeta(6)\\
&& n=1 &\\
\hline
\end{array}
$$
\caption{
\label{tab:Qi3:13}
Contribution of $1$-cylinder square-tiled surfaces to the Masur--Veech volume
$\Vol\cQ(1^3,-1^3)$
}
\end{table}

Table~\ref{tab:Qi3:13}  above  shows  that the total contribution of $1$-cylinder
separatrix   diagrams   to   the   volume   $\Vol\cQ(1^3,-1^3)$  is
$77\zeta(6)$.   The   statistics  of  frequencies of  $1:2:3$-cylinder
square-tiled surfaces  in $\Vol\cQ(1^3,-1^3)$ collected experimentally
gives           proportions
$0.4366:0.4000:0.1634$ which results in
$$
\Vol\cQ(1^3,-1^3)\approx \frac{77\zeta(6)}{0.4366}\approx 0.1866 \pi^{6}\,.
$$
as  an approximate value of the volume. The exact value of the volume
found by E.~Goujard in~\cite{Goujard:volumes} gives
$$
\Vol\cQ(1^3,-1^3)=\frac{11}{60}\cdot\pi^6\approx 0.1837 \pi^{6}\,.
$$
The types of separatrix diagrams and orders
of their symmetry groups presented in the table above
matches the calculation by means of recursive relation
considered in Example~\ref{ex:Qi313}.

\appendix


\section{Impact of the choice of the integer lattice on
diagram-by-diagram counting of Masur--Veech volumes}
\label{s:contibution:of:diag:for:two:lattices}

Recall the following two natural choices of the integer lattice in
period coordinates of a stratum of quadratic differentials.

\begin{enumerate}
\item the subset of $H^1_-(\hat S, \{\hat{P}_1,\dots, \hat{P}_\noz\};\C)$
consisting of those linear forms which take values in $\Z\oplus i\Z$
on $H_1^-(\hat S, \{\hat{P}_1,\dots,\hat{P}_\noz\};\Z)$
\item $H^1_-(\hat S, \{\hat{P}_1,\dots, \hat{P}_\noz\};\C)\cap
H^1(\hat S,\{\hat{P}_1,\dots, \hat{P}_\noz\};\Z\oplus i\Z)$
\end{enumerate}
Here we do not mark the preimages of simple poles, i.e.
$\hat{P}_1,\dots, \hat{P}_\noz$ are preimages of zeroes of
the quadratic differential under the double cover (see
Appendix in~\cite{DGZZ:meanders} for details on various
conventions). The difference between the two choices
affects the linear holonomy along saddle connections
joining two distinct zeroes. Under the first convention the
linear holonomy along such saddle connections belongs to
the half integer lattice $\frac{1}{2}\Z\oplus
\frac{i}{2}\Z$ while under the second convention it belongs
to the integer lattice $\Z\oplus i\Z$. This implies that in
genus 0 the first lattice in the period coordinates is a
proper sublattice of index $4^{s-1}$ of
the second one, where $s$ is the number of zeroes of the
quadratic differential.

\begin{figure}[htb]
\includegraphics{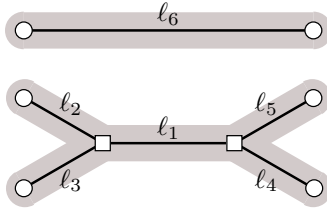}
\begin{picture}(0,0)(0,0)
\put(-9,-50){$\ell_1$}
\put(-46,-41){$\ell_2$}
\put(-46,-70){$\ell_3$}
\put(28,-70){$\ell_4$}
\put(28,-41){$\ell_5$}
\put(-9,-7){$\ell_6$}
\end{picture}

\vspace{80pt} 
\caption{
\label{fig:Q11i6}
A separatrix diagram for $\cQ(1^2, -1^6)$
}
\end{figure}

Thus, in the case of the stratum $\cQ(1^2, -1^6)$, it is a sublattice
of index $4$. Note, however, that the contributions of individual
separatrix diagrams change by the factors, which are, in general,
different from the index of one lattice in the other. Consider, for
example the separatrix diagram as in Figure~\ref{fig:Q11i6}
representing the stratum $\cQ(1^2, -1^6)$. The absolute contribution
of this separatrix diagram is twice bigger under the first choice of
the lattice than under the second one. Indeed, under the first choice
of the lattice in period coordinates, the parameter $\ell_1$ is
half-integer, as well as all the other parameters $\ell_2,\dots,
\ell_6,h,\phi$, (where $h,\phi$ are the height and the twist of the
single cylinder) whereas $\ell_1$ is integer under the second choice
of the lattice,  and the other parameters are half-integers. Hence,
the number of partitions of a given natural number $w$ (representing
the length of the waist curve of the single cylinder) into the sum
$$
w=2(\ell_1+\ell_2+\ell_3+\ell_4+\ell_5)
$$
is asymptotically twice bigger under the first choice of the lattice.

Now let us perform the computation for this diagram under the first
convention of the choice of the lattice. When the zeroes and poles
are \textit{not labeled}, the diagram has symmetry of order $4$. Since
the twist $\phi$ is half-integer, there are $2w$ choices of $\phi$.
Recall also, that the
the squares of the tiling
have side $1/2$.
Thus, under the first
choice of the lattice in period coordinates, the number of
square-tiled surfaces tiled with at most $2N$ squares
corresponding to this separatrix diagram
has the following asymptotics as $N\to+\infty$:
\begin{multline*}
\frac{1}{4}\sum_{\substack{\ell_1,\ell_2,\ell_3,\ell_4,\ell_5,h\in\N/2\\
(2(\ell_1+\ell_2+\ell_3+\ell_4+\ell_5))\cdot h\le N/2}} 2(2(\ell_1+\ell_2+\ell_3+\ell_4+\ell_5))
\sim
\frac{1}{4}\sum_{\substack{w,H\in\N\\w\cdot H\le N}} 2w\cdot\frac{w^4}{4!}
=\\=
\frac{1}{2\cdot 4!}\ \sum_{\substack{w,H\in\N\\w\le \frac{N}{H} }} w^5
\sim
\frac{1}{2\cdot 4!}\ \sum_{H\in\N} \frac{1}{6}\cdot \left(\frac{N}{H}\right)^6
=
\frac{N^6}{12\cdot 4!}\cdot\sum_{H\in\N} \frac{1}{H^6}
=
\frac{N^6}{12\cdot 4!}\cdot \zeta(6)\,.
\end{multline*}
Here in the first equivalence we passed from the half-integer
parameter $h$ to the integer parameter $H=2h$ replacing the condition
$wh\le N/2$ by the equivalent condition $wH\le N$.
Multiplying by
$\cfrac{2\cdot 6}{N^6}$ as in~\eqref{eq:volume:as:limit}
and multiplying by the factor $6!\cdot
2!$ responsible for numbering of zeroes and poles, we get
the total contribution $60\zeta(6)$ to the volume
$\Vol^{(1)}\cQ^{numbered}_1(1^2,-1^6)$ defined under the first
convention on the choice of the lattice.

Similar computations for each separatrix
diagram in this stratum are cumbersome, so,
following~\cite{AEZ:Dedicata}, we
distribute the diagrams into groups organized in the following way.

Each connected component of the separatrix diagram is encoded by a
vertex of a graph decorated with an ordered pair of natural numbers
indicating the number of zeroes and poles living at the corresponding
component. A flat cylinder joining two connected components of a
separatrix diagram is encoded by an edge of the graph. For example,
the separatrix diagram from Figure~\ref{fig:Q11i6} contains two
connected components joined by a single cylinder. The corresponding
graph contains two vertices joined by a single edge; one vertex is
decorated with the pair $(2,4)$ (standing for $2$ zeroes and $4$
poles) and the other vertex is decorated with the pair $(0,2)$
(standing for $0$ zeroes and $2$ poles). This graph is the top entry
of the left column in Table~\ref{table:normalization:Q11i6}.

\begin{table}[htb]
$$
\begin{array}{|c|c|c|}
\hline
\text{Tree} &\text{Contribution to }\Vol^{(1)}&\cfrac{\text{Contribution to }\Vol^{(1)}}{\text{Contribution to }\Vol^{(2)}}\\
[-\halfbls]
&&\\
\hline&&\\
\begin{picture}(35,10)(0,0)
\put(0,0){\circle{4}}
\put(-5,5){\tiny 2,4}
\put(2,0){\line(1,0){26}}
\put(30,0){\circle{4}}
\put(25,5){\tiny 0,2}
\end{picture}
& 60\zeta(6)&2
\\
&&\\
\hline&&\\
\begin{picture}(35,10)(0,0)
\put(0,0){\circle{4}}
\put(-5,5){\tiny 1,3}
\put(2,0){\line(1,0){26}}
\put(30,0){\circle{4}}
\put(25,5){\tiny 1,3}
\end{picture}
&80\zeta(6)&2^7
\\
&&\\
\hline&&\\
\begin{picture}(70,10)(-5,0)
\put(0,0){\circle{4}}
\put(-5,5){\tiny 0,2}
\put(2,0){\line(1,0){26}}
\put(30,0){\circle{4}}
\put(25,5){\tiny 2,2}
\put(32,0){\line(1,0){26}}
\put(60,0){\circle{4}}
\put(55,5){\tiny 0,2}
\end{picture}
&72\zeta(2)\zeta(4)&2
\\
&&\\
\hline
&&\\
\begin{picture}(70,10)(-5,0)
\put(0,0){\circle{4}}
\put(-5,5){\tiny 1,3}
\put(2,0){\line(1,0){26}}
\put(30,0){\circle{4}}
\put(25,5){\tiny 1,1}
\put(32,0){\line(1,0){26}}
\put(60,0){\circle{4}}
\put(55,5){\tiny 0,2}
\end{picture}
&48\zeta(2)\zeta(4)&2^5
\\
&&\\
\hline&&\\
\begin{picture}(100,10)(-5,0)
\put(0,0){\circle{4}}
\put(-5,5){\tiny 0,2}
\put(2,0){\line(1,0){26}}
\put(30,0){\circle{4}}
\put(25,5){\tiny 1,1}
\put(32,0){\line(1,0){26}}
\put(60,0){\circle{4}}
\put(55,5){\tiny 1,1}
\put(62,0){\line(1,0){26}}
\put(90,0){\circle{4}}
\put(85,5){\tiny 0,2}
\end{picture}
&24\zeta^3(2)&2^3
\\
&&\\
\hline
&&\\
\begin{picture}(70,10)(-5,5)
\put(0,0){\circle{4}}
\put(-5,5){\tiny 0,2}
\put(2,0){\line(1,0){26}}
\put(30,0){\circle{4}}
\put(25,5){\tiny 2,0}
\put(32,0){\line(2,1){26}}
\put(60,14){\circle{4}}
\put(65,14){\tiny 0,2}
\put(32,0){\line(2,-1){26}}
\put(60,-14){\circle{4}}
\put(65,-14){\tiny 0,2}
\end{picture}
&4\zeta^3(2)&2
\\
&&\\
&&\\
\hline
\end{array}
$$
\caption{
\label{table:normalization:Q11i6}
Table of diagram contributions to the Masur--Veech volume
$\Vol\cQ(1^2, -1^6)$ in normalizations $(1)$ and $(2)$}
\end{table}

Note that the stratum $\cQ(1^2,-1^6)$ corresponds to genus
zero, so the underlying topological surface is a sphere.
This implies that the graph defined by a separatrix diagram
representing the stratum $\cQ(1^2,-1^6)$ is a tree. The
first column of Table~\ref{table:normalization:Q11i6}
provides the list of all possible decorated trees which
appear for the stratum $\cQ(1^2,-1^6)$. It is easy to
verify that the ratio of contributions of a given
separatrix diagram to the volume of the stratum
$\cQ(1^\noz,-1^{\noz+4})$ computed under the two
conventions on the choice of the integer lattice depends
only on the corresponding decorated tree. We group together
all the diagrams corresponding to each decorated tree and
indicated in the second column the corresponding
contribution to the volume under the first choice of the
lattice (using~\cite[\S 3.8]{AEZ:Dedicata} as the source).
In the third column we give the ratio of the contributions
represented by the corresponding tree. For example, the
tree in the first line represents the unique diagram shown
in Figure~\ref{fig:Q11i6}; as it was computed above its
contribution to the volume under the first choice of the
lattice is $60\zeta(6)$ and the contribution to the volume
under the second choice of the lattice is half as small.
These data constitute the first line of
Table~\ref{table:normalization:Q11i6}.

Recall that the normalization factor between the two lattices in the
period coordinates of the stratum $\cQ(1^2,-1^6)$ is $4$. However,
observing Table~\ref{table:normalization:Q11i6} the reader can see
that the individual contributions of diagrams differ by factors $2$,
$2^3$, $2^5$, $2^7$.

Note that the trees with the same number of edges provide
contributions of the same ``arithmetic'' nature, namely the total
contribution of $1, 2, 3$-cylinder diagrams are
$$
140\zeta(6)+120\zeta(2)\zeta(4)+28\zeta^3(2)=\cfrac{\pi^6}{2}=\Vol^{(1)}\cQ(1^2,-1^6)
$$
respectively under the first choice of the lattice and
$$
\frac{245}{8}\zeta(6)+\frac{75}{2}\zeta(2)\zeta(4)+5\zeta^3(2)
=\cfrac{\pi^6}{8}=\Vol^{(2)}\cQ(1^2,-1^6)
$$
respectively under the second choice. The volumes $\Vol^{(1)}$ and
$\Vol^{(2)}$ differ by the factor $4$ as expected.

We get a polynomial identity
$$
140\zeta(6)+120\zeta(2)\zeta(4)+28\zeta^3(2)=\cfrac{\pi^6}{2}=
4\Big(\frac{245}{8}\zeta(6)+\frac{75}{2}\zeta(2)\zeta(4)+5\zeta^3(2)\Big)
$$
in zeta values at even integers. Considering other strata
$\cQ(1^\noz,-1^{\noz+4})$ we get an infinite series of analogous
identities in zeta values at even integers.

We did not study the identities resulting from different choices of
the lattice in period coordinates for more general strata of
meromorphic quadratic differentials with at most simple poles in
genus zero. Considering zeroes of even order might produce identities
of much more elaborate arithmetic nature.

If our guess that the contribution of $k$-cylinder square-tiled
surfaces to a given stratum of Abelian differentials is a polynomial
in multiple zeta values with rational (or even integer) coefficients
is true, then playing with different choices of an integer lattice we
will get infinite series of mysterious polynomial identities in
multiple zeta values.

Another challenge is to see whether one can obtain some information
about volume asymptotics for large genera playing with the choice of
an integer lattice. We leave both questions as a problem, which might
be interesting to study.

\begin{Problem}
Describe and study polynomial identities on multiple zeta values
arising from $k$-cylinder contributions to the Masur--Veech volumes
under different choices of integer lattices in period coordinates.
Study these identities in asymptotic regimes when the genus of the
surface or the number of simple poles tends to infinity.
\end{Problem}


\section{(by Philip~Engel) Square-tiled surfaces with one horizontal cylinder}
\label{Engel}

We compute the absolute contribution $c_1(\mathcal{H}(m_1,\dots,m_n))$ of the one-cylinder surfaces to the Masur--Veech volume of a stratum, using some representation theory of the symmetric group, see~\cite{james} for a general reference. Let $\nu_i=1+m_i$ so that $\nu:=\{\nu_i\}$ is a partition of $2g-2+n$. In this section, we assume the zeroes are unlabelled, unless otherwise specified.

Let $N_\nu(d)$ denote the weighted number of square-tiled surfaces in the stratum $\mathcal{H}(m_1,\dots,m_n)$ with $d$ squares, such that there is one horizontal cylinder of width $d$ and height $1$. Any such surface is a degree $d$ branched cover of a torus, ramified only over the origin, whose horizontal monodromy is a full cycle in $S_d$ and whose monodromy around the origin is of cycle type $\nu$. Write $C_\nu$ for the conjugacy class in $S_d$ with cycle type $\nu$ and $C_{cycle}$ for the conjugacy class of a $d$-cycle. Then $N_\nu(d)$ is given by the formula $$N_\nu(d)=\frac{1}{d!}\,\#\{(x,y,z)\in S_d\times C_{cycle}\times C_\nu\,:\,[x,y]z=1\}.$$ Here $x$ and $y$ are the monodromies of the fiber over a base point on the torus, with respect to vertical and horizontal loops, and $z$ denotes the monodromy of a simple loop enclosing the origin of the torus.

The irreducible representations $\rho^\lambda\,:\,S_d\rightarrow \GL(V^\lambda)$ of the symmetric group are indexed by partitions $\lambda\vdash d$. Define $\dim \lambda:=\dim V^\lambda$. Let $\chi^\lambda(g)$ denote the associated character, that is the trace of $\rho^\lambda(g)$. Then $\chi^\lambda$ depends only on the conjugacy class of $g$, uniquely determined by its cycle type. We identify $\lambda$ with its Young diagram. We say the Young diagram of $\lambda$ is {\it L-shaped} if at most one part $\lambda_i$ is not equal to $1$.

\begin{lemma}\label{one} Define an element of the group algebra $$A:= \sum_{\substack{ x\in S_d \\ y\in C_{cycle}}} [x,y]\in \C[S_d].$$ Then $A$ acts on $V^\lambda$ by the scalar $$f_A(\lambda)=\twopartdefotherwise{\frac{d!(d-1)!}{(\dim \lambda)^2}}{\lambda \textrm{ is L-shaped,}}{0}.$$  \end{lemma}

\begin{proof} From the definition, $A$ is central in the group algebra. Thus by Schur's lemma, the extension of $\rho^\lambda$ to a homomorphism $\C[S_d]\rightarrow \textrm{End}(V^\lambda)$ sends $A$ to a scalar. Similarly, the element $$A_y:=\sum_{x\in S_d}xyx^{-1}$$ is central in $\C[S_d]$. Taking the trace, it must act by $$\frac{|S_d|\chi^\lambda(y)}{\dim \lambda}id_{V^\lambda}.$$ Hence $A=\sum_{y\in C_{cycle}} A_yy^{-1}$ acts by $$\sum_{y\in C_{cycle}}\frac{|S_d|\chi^\lambda(y)}{\dim \lambda}\rho^\lambda(y^{-1}).$$ Taking traces again, we see that $A$ acts by the scalar $$\frac{|S_d| |C_{cycle}|\chi^\lambda(y)\chi^\lambda(y^{-1})}{(\dim\lambda)^2}$$ where $y$ lies in the conjugacy class of a length $d$ cycle. Noting that $y$ and $y^{-1}$ lie in the same conjugacy class, we conclude that $A$ acts by the scalar $$\frac{d!(d-1)!\chi^\lambda(y)^2}{(\dim\lambda)^2}$$ on $V^\lambda$. Finally, by the Murnaghan-Nakayama rule (\cite{james}, 21.1) $$\chi^\lambda(y)=\twopartdefotherwise{\pm 1}{\lambda\textrm{ is L-shaped}}{0}$$ The lemma follows.
   \end{proof}

\begin{definition} Let $f_\nu(\lambda)$ denote the {\it central character} --- that is the scalar that the central element $$C_\nu:=\sum_{g\in C_\nu} g\in \C[S_d]$$ acts by on $V^\lambda$. Explicitly, $f_\nu(\lambda)=|C_\nu|\frac{\chi^\lambda(g)}{\dim \lambda}$ for any $g\in C_\nu$.\end{definition}

\begin{proposition}\label{two} The generating function for the weighted number of square-tiled surfaces in $\mathcal{H}(m_1,\dots,m_n)$ with $d$ tiles and one horizontal cylinder of width $d$ and height $1$ is given by the formula $$h_\nu(q)=\sum_{\lambda\textrm{ L-shaped}} \frac{f_\nu(\lambda)}{|\lambda|}q^{|\lambda|}.$$\end{proposition}

\begin{proof} The proof is a standard argument in Hurwitz theory. Note that $$N_\nu(d)=\frac{1}{d!}\,[id] \,A\cdot C_\nu$$ where $[id]$ denotes the coefficient of the identity in the group algebra. We may extract this coefficient by taking the trace in the regular representation $V^{reg}$, because the identity is the only element acting with non-zero trace: $$N_\nu(d)=\frac{1}{(d!)^2}\chi^{reg}(A\cdot C_\nu).$$ Since $A$ and $C_\nu$ act by scalars on $V^\lambda$, the action of $A\cdot C_\nu$ respects the decomposition into isotopic components $$V^{reg}=\bigoplus_{\lambda\vdash d}\, (V^\lambda)^{\oplus \dim\lambda}\,.$$ We conclude that $$N_\nu(d)=\sum_{\lambda\vdash d} \left(\frac{\dim \lambda}{d!}\right)^2f_A(\lambda)f_\nu(\lambda).$$ The proposition then follows immediately from Lemma \ref{one}. \end{proof}

The L-shaped partitions $\lambda_{a,b}$ are indexed by pairs of positive half-integers $a,b\in\frac{1}{2}+\Z_{\geq 0}$ where $a+\tfrac{1}{2}$ and $b+\tfrac{1}{2}$ are the largest parts of $\lambda$ and $\lambda^t$ respectively. Here $\lambda^t$ denotes the transpose, gotten by reflecting the Young diagram along the line $x=y$. Note that $|\lambda|=a+b$. We conclude that $$h_\nu(q)=\sum_{a,b\in \tfrac{1}{2}+\Z_{\geq 0}}\frac{f_\nu(\lambda_{a,b})}{a+b}q^{a+b}.$$ A result of Kerov-Olshanski \cite{kerov} states that $f_\nu(\lambda)$ is a {\it shifted-symmetric polynomial}. That is, if one orders the parts $\lambda=\{\lambda_1\geq \lambda_2\geq \lambda_3\geq \dots\}$, then $f_\nu$ a polynomial symmetric in the variables $\lambda_i-i$. The algebra of shifted symmetric polynomials is denoted $\Lambda^*$ and is freely generated by {\it shifted power-sums} $$p_k(\lambda):=\sum_{i=1}^{\infty}(\lambda_i-i+\tfrac{1}{2})^k-(-i+\tfrac{1}{2})^k.$$ Define the {\it degree grading} by declaring $\deg p_k=k$ and extend this to a grading on $\Lambda^*$. Note that this grading on $\Lambda^*$ differs from the weight grading defined in \cite{Eskin:Okounkov:Inventiones}, which declares $\textrm{wt}\, p_k=k+1$. Then Theorem 5 of \cite{kerov} implies that $$LT(f_\nu)=\frac{1}{|\textrm{Aut}(\nu)|}\prod_{i=1}^{\ell(\nu)}\frac{p_{\nu_i}}{\nu_i},$$ where $LT$ denotes the leading term of $f_\nu$ with respect to our degree grading. We have a telescoping sum $$p_k(\lambda_{a,b})=a^k-(-b)^k.$$ Thus, we conclude that $$h_\nu(q)=\frac{1}{|\textrm{Aut}(\nu)|\prod \nu_i}\,\,\sum_{d\geq 1} \frac{q^d}{d} \!\!\!\!\sum_{\substack{a+b=d \\ a,b\in\tfrac{1}{2}+\Z_{\geq 0}}} \prod_{i=1}^{\ell(\nu)}\,(a^{\nu_i}-(-b)^{\nu_i})+(\textrm{lower order terms})\,,$$ where the lower order terms are various homogenous polynomials in $a$ and $b$ of degree less than $\sum \nu_i$. Next, observe that $$\frac{1}{d}\!\!\sum_{\substack{a+b=d \\ a,b\in\tfrac{1}{2}+\Z_{\geq 0}}} \prod_{i=1}^{\ell(\nu)}\,\left(\left(\frac{a}{d}\right)^{\nu_i}-\left(-\frac{b}{d}\right)^{\nu_i}\right)$$ is a Riemann sum of mesh width $1/d$ approximating the integral $$I(\nu):=\int_0^1\,\,\prod_{i=1}^{\ell(\nu)}(x^{\nu_i}-(x-1)^{\nu_i})\,dx,$$ whereas the lower order terms are similarly Riemann sums of mesh width $1/d$ approximating integrals of lower degree. Asymptotically as $d\rightarrow \infty$, the Riemann sum converges to the integral. We conclude that as $q\rightarrow 1$, $$h_\nu(q) \sim \frac{I(\nu)}{|\textrm{Aut}(\nu)|\prod\nu_i}\,\,\sum_{d\geq 1} q^d\,d^{2g-2+n},$$ assuming that the integral $I(\nu)$ is positive, as otherwise lower order terms would become relevant. The integral is in fact positive, because the integrand is. The number of $i$ for which $\nu_i$ is even must itself be even.

Let $H_\nu(q)$ denote the generating function for all square-tiled surfaces with one horizontal cylinder in the stratum $\mathcal{H}(m_1,\dots,m_n)$, regardless of the height of the cylinder. The width is a divisor of the total number of squares, and we may rescale the height to produce a square-tiled surface with one horizontal cylinder and height $1$. We conclude that $$H_\nu(q)\sim \frac{I(\nu)}{|\textrm{Aut}(\nu)|\prod \nu_i}\sum_{d\geq 1}q^d \sigma_{2g-2+n}(d)$$ where $\sigma_k(d)$ is the divisor power sum. Thus, we have \begin{theorem}\label{three} The absolute $1$-cylinder contribution to the Masur--Veech volume of $\mathcal{H}:=\mathcal{H}(m_1,\dots,m_n)$, with the zeroes ordered, is $$c_1(\mathcal{H})=2\, \frac{I(\nu)}{\prod \nu_i}\zeta(\dim(\mathcal{H})).$$ \end{theorem}

Note that we multiply by a factor of $|\textrm{Aut}(\nu)|$ to order the zeroes. We now compute two examples, to verify agreement with Corollary \ref{cor:total:contribution:min:and:principal}.

\begin{example} Consider the principal stratum $\mathcal{H}(1,\dots,1)$. That is $\nu_i=2$ for all $i=1,\dots,2g-2$. Then $$\frac{I(\nu)}{\prod \nu_i}= \int_0^1\,\left(\frac{x^2-(1-x)^2}{2}\right)^{2g-2}\,dx=\int_0^1 (x-\tfrac{1}{2})^{2g-2}\, dx= \frac{2^{2-2g}}{2g-1}.$$ Thus the $1$-cylinder contribution to the volume is $$\frac{2^{3-2g}}{2g-1}\zeta(4g-3).$$ For the minimal stratum $\mathcal{H}(2g-2)$, we have $$\frac{I(\nu)}{\prod \nu_i} =\frac{1}{2g-1}\int_0^1 x^{2g-1}+(1-x)^{2g-1}\, dx = \frac{1}{g(2g-1)}$$ and thus the $1$-cylinder contribution is $$\frac{2\zeta(2g)}{g(2g-1)}.$$
\end{example}

\begin{proposition}\label{four} Asymptotically as the genus grows, $$I(\nu)\sim \frac{2}{\sum \nu_i}.$$ \end{proposition}

\begin{proof} Let $d=\sum \nu_i$. First, observe that the integrand of $I(\nu)$ is even about $1/2$, and thus, we may write $$I(\nu)=2\int_{1/2}^1 x^d\prod_{i=1}^{\ell(\nu)} (1-(1-x^{-1})^{\nu_i})\,dx.$$ Observe that $x^{\nu_i}-(x-1)^{\nu_i}$, and thus the whole integrand, is non-negative and monotonically increasing between $\frac{1}{2}$ and $1$. We show that the integral concentrates in a small neighborhood of $1$. Define $$R_\nu(x):=\prod_{i=1}^{\ell(\nu)} (1-(1-x^{-1})^{\nu_i}).$$  Let $u(\alpha)=\frac{\alpha d}{\alpha d+1}$. For some fixed $d$, the quantity $R_\nu(u(\alpha))$ is either maximized or minimized when all $\nu_i$ are minimal and of the same parity, the parity condition ensuring that all terms are less than or greater than one. Since $\nu_i\geq 2$, we have the bounds $$e^{-1/(2\alpha^2d)}\leq (1-1/(\alpha d)^2)^{d/2} \leq R_\nu(u(\alpha))\leq (1+1/(\alpha d)^2)^{d/2}\leq e^{1/(2\alpha^2d)}.$$ These bounds rapidly approach $1$ as $d\rightarrow \infty$. On the other hand, we have $$u(\alpha)^d=(1-1/(\alpha d+1))^{d}\approx e^{-1/\alpha}$$ and thus the value of the integrand of $I(\nu)$ at $u(\alpha)$ decays as $\alpha$ approaches zero. For instance, setting $\alpha=\frac{1}{N\log(d)}$ gives the bound $d^{-N}$. The monotonicity of the integrand then implies $$\lim_{d\rightarrow \infty} P(d)\int_{1/2}^{u(\alpha)} x^d R_\nu(x)\,dx= 0$$ for any polynomial $P(d)$ of degree less than $N$. Now we compute the remaining integral from $u(\alpha)$ to $1$. The bounds on $R_\nu(u(\alpha))$ only get better as $\alpha$ increases, and thus, we conclude that the integrand of $I(\nu)$ is very nearly equal to $x^d$ when $x\in(u(\alpha),1)$. Therefore as $d\rightarrow \infty$, we have $$I(\nu)\sim 2\int_{u(\alpha)}^1x^d\,dx,$$ so long as this integral has only inverse polynomial decay in $d$ of degree less than $N$ (otherwise the integral from $1/2$ to $u(\alpha)$ would be relevant in the asymptotic). Integrating, we find that it does whenever $N>1$: $$I(\nu)\sim \frac{2}{d+1}\sim \frac{2}{\sum \nu_i}.$$ The proposition follows. \end{proof}

From Theorem \ref{three} and Proposition \ref{four}, we have $$c_1(\mathcal{H}(m_1,\dots,m_n))\sim \frac{4}{\sum (m_i+1)\prod(m_i+1)},$$ providing an alternative proof of Corollary~\ref{th:contribution:1:cyl:g:to:infty}.

\end{document}